\documentclass[12pt]{article}

\usepackage{amsmath,amssymb,amsthm}
\usepackage[pagebackref,colorlinks=true,allcolors=black,bookmarksopen,bookmarksdepth=3]{hyperref}
\usepackage[margin=1in]{geometry}
\usepackage{tikz}

\usepackage{enumitem}
\setitemize{label=$\bullet$,topsep=1pt,itemsep=1pt,partopsep=0pt,parsep=0pt}

\newtheorem{theorem}{Theorem}[section]
\newtheorem{lemma}[theorem]{Lemma}
\newtheorem{corollary}[theorem]{Corollary}
\newtheorem{proposition}[theorem]{Proposition}

\theoremstyle{definition}
\newtheorem{definition}[theorem]{Definition}

\theoremstyle{remark}
\newtheorem{remark}[theorem]{Remark}

\newtheorem{convention}[theorem]{Convention}

\numberwithin{equation}{section}

\newcommand{\F}{\mathfrak F}

\newcommand{\ZZ}{\mathbb Z}
\newcommand{\CC}{\mathbb C}
\newcommand{\RR}{\mathbb R}
\newcommand{\HH}{\mathbb H}
\newcommand{\sL}{\mathcal L}
\newcommand{\A}{\mathfrak A}
\newcommand{\B}{\mathfrak B}
\newcommand{\id}{\operatorname{id}}

\newcommand{\Hom}{\operatorname{Hom}}
\newcommand{\Con}{\operatorname{Con}}

\newcommand{\Diff}{\operatorname{Diff}}
\newcommand{\Isom}{\operatorname{Isom}}
\newcommand{\isom}{\operatorname{\mathfrak{isom}}}

\newcommand{\Teich}{\operatorname{Teich}}
\newcommand{\Homeq}{\operatorname{Homeq}}

\newcommand{\Stab}{\operatorname{Stab}}
\newcommand{\Nbd}{\operatorname{Nbd}}
\newcommand{\colim}{\operatornamewithlimits{colim}}
\newcommand{\acts}{\curvearrowright}
\newcommand{\PGL}{\operatorname{PGL}}
\newcommand{\GL}{\operatorname{GL}}
\mathchardef\mhyphen"2D
\newcommand{\free}{\mathrm{free}}
\newcommand{\refl}{\mathrm{refl}}
\newcommand{\trefl}{\mathrm{trefl}}
\newcommand{\tame}{\mathrm{tame}}
\newcommand{\wild}{\mathrm{wild}}

\newcommand{\Maps}{\operatorname{Maps}}
\newcommand{\Bun}{\operatorname{Bun}}
\newcommand{\tildetimes}{\mathbin{\tilde\times}}
\newcommand{\hacts}{\mathrel{\smash{\overset h\acts}}}
\newcommand{\hto}{\mathrel{\smash{\xrightarrow h}}}

\begin{document}

\title{Smoothing finite group actions on three-manifolds}

\author{John Pardon}

\date{March 25, 2020}

\maketitle

\begin{abstract}
We show that every continuous action of a finite group on a smooth three-manifold is a uniform limit of smooth actions.
\end{abstract}

\section{Introduction}

Every continuous finite group action on a manifold of dimension $\leq 2$ is conjugate to a smooth action \cite[pp340--341]{edmonds}\cite{constantinkolev,eilenberg,brouwer,kerekjarto}.
In contrast, there are many examples of finite group actions on three-manifolds which are not conjugate to smooth actions, see Bing \cite{binginvolution,bingperiodic}, Montgomery--Zippin \cite{montgomeryzippin}, and Alford \cite{alford}; all of these examples are defined as uniform limits of smooth actions.

In this paper, we show that every continuous action of a finite group on a smooth three-manifold is a uniform limit of smooth actions, answering an old question (see Edmonds \cite[p343]{edmonds}).
Recall that a neighborhood of an action $\varphi:G\acts M$ in the uniform topology (aka the strong $C^0$ topology) consists of those actions $\tilde\varphi:G\acts M$ such that $(\varphi(g)x,\tilde\varphi(g)x)\in U$ for every $(g,x)\in G\times M$, where $U\subseteq M\times M$ is a neighborhood of the diagonal.
Note that we do not assume that $M$ is compact.

\begin{theorem}\label{main}
Every continuous action $\varphi:G\acts M$ of a finite group on a smooth three-manifold is a uniform limit of smooth actions $\tilde\varphi:G\acts M$.
If $\varphi$ is smooth over $\Nbd K$ for $K\subseteq M$ closed and $\varphi(G)$-invariant, then we may take $\tilde\varphi=\varphi$ over $\Nbd K$.
\end{theorem}

\begin{remark}
In higher dimensions, there exist (even free) finite group actions which are not uniformly approximable by smooth actions.
Indeed, let $M$ be a topological manifold of dimension $\geq 5$ which admits no smooth structure but which admits a finite cover $\widetilde M\to M$ with $\widetilde M$ admiting a smooth structure (e.g.\ $M$ could be a non-smoothable fake real projective space \cite{lopez}).
By passing to a further cover, we may assume that $\widetilde M\to M$ is Galois, say with Galois group $G$.
Choosing arbitrarily a smooth structure on $\widetilde M$, we claim that the continuous free action $\varphi:G\acts\widetilde M$ is not uniformly approximable by smooth actions.
Indeed, for any action $\varphi':G\acts\widetilde M$ sufficiently uniformly close to $\varphi$, the quotients $M=\widetilde M/\varphi(G)$ and $\widetilde M/\varphi'(G)$ are homeomorphic by the Chapman--Ferry $\alpha$-approximation theorem \cite{chapmanferry} (we thank Mladen Bestvina for pointing this out to us).
Hence $\varphi'$ cannot be smooth, as otherwise we would obtain a smooth structure on $M$, which we have assumed does not exist.
\end{remark}

We now sketch the proof of Theorem \ref{main}, starting with the case of free actions.
If $\varphi:G\acts M$ is free, then the quotient space $M/G$ is a topological manifold.
By Bing and Moise, there exists a smooth structure on $M/G$, which we can pull back to a smooth structure on $M$ (call it $M^s$) with respect to which $\varphi$ is smooth.
Now the identity map $\id:M\to M^s$ is a homeomorphism between smooth three-manifolds, and Bing and Moise tell us that any homeomorphism between smooth three-manifolds can be uniformly approximated by diffeomorphisms.
Denoting by $\alpha:M\to M^s$ such a diffeomorphism, we conclude that the conjugated action $\alpha^{-1}\varphi\alpha:G\acts M$ is smooth and uniformly close to $\varphi$.
In fact, this reasoning shows moreover that any action $\varphi:G\acts M$ can be smoothed over the (necessarily open) locus where it is free.

To treat more general actions $\varphi:G\acts M$, we need some understanding of which subsets of $M$ can occur as the fixed points of the action of $G$ or of one of its subgroups.
Smith theory concerns precisely this question, and provides that for any homeomorphism $g$ of prime order $p$ of a topological three-manifold $M$, the fixed set $M^g$ is a topological manifold (of possibly varying dimension and possibly wildly embedded inside $M$).
Writing
\begin{equation}
M^g=M^g_{(0)}\sqcup M^g_{(1)}\sqcup M^g_{(2)}\sqcup M^g_{(3)}
\end{equation}
for the decomposition of $M^g$ by dimension, we furthermore have that $M^g_{(2)}$ can be non-empty only when $p=2$ and $g$ reverses orientation near $M^g_{(2)}$.

The proof of Theorem \ref{main} now proceeds in three steps which smooth a given action $\varphi:G\acts M$ over successively larger open subsets of $M$.
We may assume without loss of generality that $\varphi:G\acts M$ is \emph{generically free}, namely no nontrivial element $g\in G$ acts as the identity on a nonempty open subset of $M$ or, equivalently, $M^g_{(3)}=\varnothing$ for every prime order $g\in G$.

The first step is to smooth the action over the open set $M^\free\subseteq M$ where it is free.
As discussed above, this is a straightforward application of the smoothing theory for homeomorphisms of three-manifolds due to Bing and Moise.

The second step is to smooth the action over the open set $M^\refl\subseteq M$ defined as the set of points $x$ whose stabilizer $G_x$ is either trivial or of order two, generated by an involution $g$ for which $x\in M^g_{(2)}$.
Smoothing an involution fixing a surface is essentially due to Craggs \cite{craggs}.
The main point is that any (possibly wildly) embedded surface (in particular, $F^\refl:=\{x\in M^\refl\,|\,G_x=\ZZ/2\}$) in a three-manifold can be approximated uniformly by tamely embedded surfaces (due to Bing) and that such approximations are unique up to small isotopy (due to Craggs).

The third and final step (which constitutes the main content of this paper) is to smooth the action over the remainder
\begin{equation}\label{badlocus}
M\setminus M^\refl=\bigcup_{\begin{smallmatrix}g\in G\setminus\{1\}\\g^p=1\end{smallmatrix}}M^g_{(0)}\cup M^g_{(1)}.
\end{equation}
Since this locus is a union of $0$- and $1$-dimensional manifolds (possibly wildly) embedded in $M$, it has covering dimension $\leq 1$, and this will be crucial to our argument.
We consider a small closed $G$-invariant neighborhood $M_0$ of \eqref{badlocus} with smooth boundary, and we fix a $G$-equivariant finite open cover $M_0=\bigcup_iU_i$ by small open sets $U_i$ (possibly permuted by the action of $G$) such that all triple intersections are empty ($U_i\cap U_j\cap U_k=\varnothing$ for distinct $i,j,k$).
We now find properly embedded incompressible surfaces $F_{ij}\subseteq U_i\cap U_j$ (separating the $U_i$ end and the $U_j$ end) which are $G$-invariant up to isotopy.
The construction of such surfaces uses the ``lattice of incompressible surfaces'' from \cite[\S 2]{pardonhilbertsmith} and the elementary fact that a finite group acting on a nonempty lattice always has a fixed point (take the least upper bound of any orbit).
These surfaces $F_{ij}$ divide $M_0$ into pieces $N_i$ (each a compact three-manifold-with-boundary), and $G$ acts on $\bigsqcup_iN_i$ up to homotopy.
Finally, we note that these homotopy actions can be upgraded to strict actions (by diffeomorphisms) by appealing to the JSJ decomposition, the existence of hyperbolic structures due to Thurston, the rigidity results of Mostow, Prasad, and Marden, and the solution to the Nielsen realization problem for surfaces by Kerckhoff.
The resulting smooth action of $G$ on $M$ can be made arbitrarily close to the original action in the uniform topology by taking the neighborhood $M_0$ and the open sets $U_i$ to be sufficiently small.

\begin{remark}
We work throughout this paper in the smooth category unless the contrary is explicitly stated (as in `topological manifold' or `continuous action'), however one could just as well work instead in the piecewise-linear category.
In particular, Theorem \ref{main} is equivalent to the corresponding statement in the piecewise-linear category.
\end{remark}

\begin{convention}
Manifold means Hausdorff, locally Euclidean, and paracompact.
\end{convention}

\subsection{Acknowledgements}

This paper owes much to the work of the two anonymous referees, both of whose comments contributed significantly to its accuracy and clarity.
The author gratefully acknowledges helpful communication with Ian Agol, Mladen Bestvina, Martin Bridson, Dave Gabai, and Shmuel Weinberger.

This research was conducted during the period the author served as a Clay Research Fellow.  The author was also partially supported by a Packard Fellowship and by the National Science Foundation under the Alan T.\ Waterman Award, Grant No.\ 1747553.

\section{Nielsen realization for some three-manifolds}\label{secnielsen}

This section collects various known results in three-manifold topology.
Specifically, we study the problem of upgrading a homotopy action on a three-manifold to a genuine action.
Conditions under which this is possible are well-known due to work of Jaco--Shalen \cite{jacoshalen}, Johannson \cite{johannson}, Thurston \cite{morgan}, Mostow \cite{mostow}, Prasad \cite{prasad}, Marden \cite{marden}, Kerckhoff \cite{kerckhoff}, Zimmermann \cite{zimmermannnielsen}, and Heil--Tollefson \cite{heiltollefsonII,heiltollefsonIII}.
We include this section mainly to make this paper self-contained, as we were unable to find the exact statement we need in the literature; in particular, we do not want to restrict to three-manifolds with incompressible boundary.

\subsection{Groups of diffeomorphisms and homotopy equivalences}

For a compact manifold-with-boundary $M$, we denote by $\Diff(M)$ the group of diffeomorphisms of $M$, and we denote by $\Diff(M,\partial M)$ the subgroup of those diffeomorphisms which are the identity over $\partial M$.
There is an exact sequence
\begin{equation}\label{diffses}
1\to\Diff(M,\partial M)\to\Diff(M)\to\Diff(\partial M)
\end{equation}
with the image of the final map being a union of connected components.
Similarly, we denote by $\Homeq(M)$ the monoid of self homotopy equivalences of the pair $(M,\partial M)$, and we define $\Homeq(M,\partial M)$ by the exact sequence
\begin{equation}\label{homeqses}
1\to\Homeq(M,\partial M)\to\Homeq(M)\to\Homeq(\partial M)
\end{equation}
Both \eqref{diffses} and \eqref{homeqses} induce long exact sequences of homotopy groups (with the caveat that the very final map on $\pi_0$ need not be surjective).
There is an obvious forgetful map from \eqref{diffses} to \eqref{homeqses}, which induces a map between the associated long exact sequences.
Note that $\Homeq(M)$ does \emph{not} denote the monoid of self homotopy equivalences of $M$, which is instead homotopy equivalent to $\Homeq(M^\circ)$.

Most of the spaces we will consider here are (disjoint unions of) $K(\pi,1)$ spaces, so for future use we record here the straightforward fact that the space of maps between two such spaces admits a natural group theoretic description.

\begin{lemma}\label{mapskpione}
The components of $\Maps(K(\pi,1),K(\pi',1))$ are indexed by the orbits of the conjugation action $\pi'\acts\Hom(\pi,\pi')$, and the component of a given $f:\pi\to\pi'$ is $K(Z_{\pi'}(f(\pi)),1)$ where $Z_G(H)$ denotes the centralizer of the subgroup $H\leq G$.
\qed
\end{lemma}

More succinctly, $\Maps(K(\pi,1),K(\pi',1))$ is the homotopy quotient of $\Hom(\pi,\pi')$ by the conjugation action of $\pi'$.

\subsection{Homotopy group actions}

Let $G$ be a finite group.
A (strict) action $\varphi:G\acts M$ is simply a homomorphism $\varphi:G\to\Diff(M)$.
A(n often much) weaker notion is that of a homomorphism $\varphi:G\to\pi_0\Diff(M)$ or to $\pi_0\Homeq(M)$.
In this paper, the intermediate notion of a `homotopy action' $G\hacts M$ or a `homotopy homomorphism' $G\hto\Diff(M)$ (or to $\Homeq(M)$) will play an important role.

For any topological monoid $A$ (such as $\Diff(M)$ or $\Homeq(M)$), a \emph{homotopy homomorphism} $\varphi:G\hto A$ is, by definition, a collection of maps $\varphi^k:G^{k+1}\times[0,1]^k\to A$ for all $k\geq 0$ satisfying
\begin{align}
\varphi^k(g_0,\ldots,g_k)_{[0,1]^i\times\{1\}\times[0,1]^{k-i-1}}&=\varphi^{k-1}(g_0,\ldots,g_ig_{i+1},\ldots,g_k)\\
\varphi^k(g_0,\ldots,g_k)_{[0,1]^i\times\{0\}\times[0,1]^{k-i-1}}&=\varphi^i(g_0,\ldots,g_i)\circ\varphi^{k-i-1}(g_{i+1},\ldots,g_k)
\end{align}
for $0\leq i<k$ (compare Sugawara \cite[\S 2]{sugawara} and Boardman--Vogt \cite[Definition 1.14]{boardmanvogt}).
Note the lack of any condition on the value of $\varphi^k$ when one of the inputs is the identity.

A \emph{homotopy action by diffeomorphisms} (resp.\ \emph{homotopy equivalences}) $G\hacts M$ shall mean a homotopy homomorphism $G\hto\Diff(M)$ (resp.\ $\Homeq(M)$).
We will often shorten this to `homotopy action' when it is either clear from context or irrelevant whether we mean by diffeomorphisms or by homotopy equivalences.

We endow the spaces of strict actions $\Hom(G,A)\subseteq\Maps(G,A)$ and homotopy actions $\Hom^h(G,A)\subseteq\prod_{k\geq 0}\Maps(G^{k+1}\times[0,1]^k,A)$ both with the subspace topology.
Note that if a map of topological monoids $A\to B$ is a homotopy equivalence (of spaces), then the induced map $\Hom^h(G,A)\to\Hom^h(G,B)$ is a homotopy equivalence.

There is an evident inclusion
\begin{equation}\label{strictintohomotopy}
\Hom(G,A)\hookrightarrow\Hom^h(G,A)
\end{equation}
of strict homomorphisms into homotopy homomorphisms, by taking $\varphi^k$ to be independent of the $[0,1]^k$ factor for all $k$.
Our main aim in this section is to show that in many cases, this map is a homotopy equivalence, or at least surjective on connected components.
(We will allow ourselves to not worry too much about the distinction between weak homotopy equivalence and homotopy equivalence in our discussion.)

There is another perspective on strict and homotopy actions which we will frequently take advantage of.
Let $BG=K(G,1)$ denote the classifying space of the finite group $G$, namely $BG$ is a connected space with a basepoint $\ast\in BG$ with $\pi_1(BG,\ast)=G$ (a specified isomorphism) and $\pi_i(BG,\ast)=0$ for $i>1$.
Now the data of a homotopy action by diffeomorphisms $G\hacts M$ is equivalent (as we are about to see) to the data of a bundle over $BG$ together with an identification of the fiber over the basepoint with $M$ (by this we mean that $\Hom^h(G,\Diff(M))$ is homotopy equivalent to the classifying space of such bundles), and a strict action $G\acts M$ is such a bundle equipped with a flat connection.
\emph{The problem of upgrading a homotopy action to a strict action may thus be viewed as the problem of constructing a flat connection on a given bundle over $BG$ with fiber $M$.}

To make this discussion precise, let us fix the following model of $BG$ (the specific choice of model is, of course, ultimately irrelevant).
Let $EG$ be the complete semi-simplicial complex on the vertex set $G$, i.e.\ the space $EG$ is built by gluing together $k$-simplices indexed by ordered $(k+1)$-tuples of elements of $G$ (which are the vertices of the simplex) via the obvious maps forgetting vertices (note that we do not collapse any `degenerate simplices').
This space $EG$ is contractible, and it carries a free action of $G$ by multiplication on the left.
The $k$-simplices of the quotient $BG=EG/G$ are thus indexed by $k$-tuples of elements of $G$, and the faces of the $k$-simplex $(g_1,\ldots,g_k)$ are given by $(g_2,\ldots,g_k),(g_1g_2,g_3,\ldots,g_k),(g_1,g_2g_3,g_4,\ldots,g_k),\ldots,(g_1,\ldots,g_{k-2},g_{k-1}g_k),(g_1,\ldots,g_{k-1})$.

In order to discuss bundles over $BG$, fix for every face inclusion $i:\Delta^k\hookrightarrow\Delta^n$ a germ of a retraction $r(i):\Delta^n\to\Delta^k$, such that $r(i\circ j)=r(j)\circ r(i)$.
Such retractions may be constructed by induction (the key fact used in this induction is that any smooth real valued function on the boundary of a manifold-with-corners admits smooth extensions to the interior).
Bundles over $BG$ are now defined via transition functions which are pulled back under these retractions on a neighborhood of the boundary of every simplex; the same pullback condition is also imposed on connections.

Let us denote by $\Bun((BG,*),\Diff(M))$ the space of bundles over $BG$ with marked fiber $M$ over the basepoint, equipped with a connection (up to isomorphism respecting the marking and the connection).
This space $\Bun((BG,*),\Diff(M))$ classifies families of bundles over $BG$ with marked fiber $M$ over the basepoint (with or without a connection, as a connection is a contractible choice).
The subspace of $\Bun((BG,*),\Diff(M))$ consisting of those bundles whose connection is flat is \emph{equal} to $\Hom(G,\Diff(M))$.

To compare $\Bun((BG,*),\Diff(M))$ with $\Hom^h(G,\Diff(M))$, let us recall the Adams family of paths \cite{adamscobar}.
This is, for every $k\geq 0$, a map $\gamma^k$ from the cube $[0,1]^k$ to the space of paths in $\Delta^{k+1}$ from vertex $0$ to vertex $k+1$.
These families of paths $\gamma^k$ have the following concatenation/restriction properties for $0\leq i<k$: (1) $\gamma^k|_{[0,1]^i\times\{1\}\times[0,1]^{k-i-1}}$ is the pushforward of $\gamma^{k-1}$ under the inclusion $\Delta^k\to\Delta^{k+1}$ missing vertex $i+1$, and (2) $\gamma^k|_{[0,1]^i\times\{0\}\times[0,1]^{k-i-1}}$ is the concatenation of $\gamma^i$ and $\gamma^{k-i-1}$, viewed as paths from vertex $0$ to vertex $i+1$ contained in the initial $\Delta^{i+1}\subseteq\Delta^{k+1}$ and paths from vertex $i+1$ to vertex $k+1$ contained in the final $\Delta^{k-i}\subseteq\Delta^{k+1}$, respectively.
These compatibility properties immediately imply that integrating connections along the Adams family of paths defines a map
\begin{equation}
\Bun((BG,*),\Diff(M))\to\Hom^h(G,\Diff(M)).
\end{equation}
To see that this map is a homotopy equivalence, we can show by induction on $k\geq 0$ that the analogous map from bundles over the $(k+1)$-skeleton of $BG$ to homotopy actions defined up to level $k$ is a homotopy equivalence.
When we increment $k$, the domain and codomain of this map are replaced by the total spaces of fibrations over them with fiber either empty or $\Omega^k\Diff(M)$ (the $k$-fold based loop space).
Analyzing the map on fibers explicitly, one sees that it is indeed a homotopy equivalence.

\subsection{Local stability of strict actions}

When studying families of strict actions, the following local stability result is fundamental.

\begin{proposition}\label{closeconjugate}
Fix a strict action $\varphi:G\acts M$.
For every strict action $\varphi':G\acts M$ sufficiently close to $\varphi$ in the smooth topology, there exists a diffeomorphism $\rho$ of $M$ such that $\varphi'=\rho^{-1}\varphi\rho$.
Moreover, $\rho$ may be taken to depend smoothly on $\varphi'$.
\end{proposition}

\begin{proof}
The strategy is to consider the identity map $M_{\varphi'}\to M_\varphi$ and try to correct it to make it $G$-equivariant, thus giving the desired $\rho$.
(Here $M_\varphi$ indicates $M$ equipped with the $G$-action $\varphi$, and the same for $M_{\varphi'}$.)
The obvious way to correct a map to be equivariant is by averaging, though of course this does not make sense \emph{a priori} since $M$ does not have any sort of linear structure.
We can, however, choose an atlas of local charts on $M$, each of which has linear structure, and then do our averaging locally in each chart.
Here are the details.

To begin with, recall that every strict action $\varphi:G\acts M$ is locally linear, in the sense that $M$ can be covered by $G$-equivariant open inclusions $G\times_H\RR^n\to M_\varphi$ with $H\leq G$ and $H$ acting linearly on $\RR^n$.
Indeed, to construct such a chart near a point $p\in M$, start with any locally defined map $(M,p)\to(T_pM,0)$ whose derivative at $p$ is the identity, and average it under the action of the stabilizer group $G_p$ to make it $G_p$-equivariant.

Now fix a cover of $M_\varphi$ by charts $u_i:G\times_{H_i}\RR^{n_i}\to M_\varphi$.
For each such chart, choose a $\varphi(G)$-invariant function $\eta_i:M_\varphi\to\RR_{\geq 0}$ supported inside the image of $u_i$, such that every $x\in M$ has a neighborhood over which some $\eta_i$ is $\equiv 1$ and only finitely many $\eta_i$ are not $\equiv 0$.

Now for $\varphi'$ sufficiently uniformly close to $\varphi$ and any function $f:M_{\varphi'}\to M_\varphi$ sufficiently uniformly close to the identity, we may define the $i$th averaging $A_if:M_{\varphi'}\to M_\varphi$ by the formula
\begin{equation}
(A_if)(x)=(1-\eta_i(x))\cdot f(x)+\eta_i(x)\cdot\frac 1{\left|G\right|}\sum_{g\in G}\varphi(g^{-1})f(\varphi'(g)x)
\end{equation}
for $x\in u_i(G\times_H\RR^n)$ and $(A_if)(x)=f(x)$ otherwise.
Note that $A_if$ is $G$-equivariant at any $x\in M_{\varphi'}$ at which either $f$ is $G$-equivariant or $\eta_i(\varphi'(G)x)=1$.

Now for $\varphi'$ sufficiently uniformly close to $\varphi$, the infinite composition $\rho:=(\cdots\circ A_2\circ A_1)\id:M_{\varphi'}\to M_\varphi$ is defined and $G$-equivariant.
If $\varphi'$ is sufficiently smoothly close to $\varphi$, this $\rho$ is a diffeomorphism.
\end{proof}

\begin{corollary}\label{stricthascollar}
For any strict action $\varphi:G\acts M$ on a manifold-with-boundary with $\partial M$ compact, there exists a $\varphi(G)$-equivariant boundary collar $\partial M\times[0,\varepsilon)\hookrightarrow M$.
\end{corollary}

\begin{proof}
Choose any function $r:M\to\RR_{\geq 0}$ vanishing transversely along $\partial M$.
By averaging $r$, we may make it $\varphi(G)$-invariant.
Choose any collar $\partial M\times[0,\varepsilon)\hookrightarrow M$ with second coordinate $r$.
By restricting $\varphi$ to the slices $\partial M\times\{t\}$, we obtain a family of actions $\psi_t:G\acts\partial M$ for $t\in[0,\varepsilon)$.
Apply Proposition \ref{closeconjugate} to $\psi_t$ as a perturbation of $\psi_0$ to obtain diffeomorphisms $\rho_t$ (defined for $t$ sufficiently small) such that $\psi_t=\rho_t^{-1}\psi_0\rho_t$.
Now precompose the collar with $(x,t)\mapsto(\rho_t^{-1}(x),t)$.
\end{proof}

\begin{corollary}\label{restrictionfibration}
The restriction map $\Hom(G,\Diff(M))\to\Hom(G,\Diff(\partial M))$ is a fibration.
\end{corollary}

\begin{proof}
Given a strict action $\varphi:G\acts M$ and a one-parameter family of strict actions $\psi_t:G\acts\partial M$ for $t\in[0,1]$ with $\psi_0=\varphi|_{\partial M}$, Proposition \ref{closeconjugate} guarantees that $\psi_t=\rho_t^{-1}\psi_0\rho_t$ for some one-parameter family of diffeomorphisms $\rho_t:\partial M\to\partial M$ starting at $\rho_0=\id$.
(\emph{A priori} Proposition \ref{closeconjugate} provides this only for small $t$, however a compactness argument pushes this to all $t$.)
Extending $\rho_t$ to a family of diffeomorphisms of $M$ (which we can do since $\Diff(M)\to\Diff(\partial M)$ is a fibration), we obtain a one-parameter family $\varphi_t:=\rho_t^{-1}\varphi\rho_t$.
This construction works well in families of pairs $(\varphi,\{\psi_t\}_{t\in[0,1]})$, which is enough.
\end{proof}

\subsection{Actions on circles and surfaces}

Before discussing homotopy actions on three-manifolds, we must discuss actions on circles and surfaces, where we have a good understanding, due most significantly to the solution of the Nielsen realization problem by Kerckhoff \cite{kerckhoff} and again later by Wolpert \cite{wolpert} (other than the appeal to their seminal work, the reasoning in this section is essentially elementary).

\begin{convention}\label{connecteddisconnected}
For sake of linguistic convenience, we tacitly assume all manifolds being acted on to be connected.
The results and arguments all extend trivially to the general case, which we will in fact need.
\end{convention}

In the case of the circle, we have homotopy equivalences
\begin{equation}
S^1\rtimes(\ZZ/2)=\Isom(S^1)\xrightarrow\sim\Diff(S^1)\xrightarrow\sim\Homeq(S^1).
\end{equation}
It follows that a homomorphism $G\to\pi_0\Diff(S^1)=\pi_0\Homeq(S^1)$ is simply a homomorphism $G\to\ZZ/2$ recording which elements of $G$ reverse orientation.
It also follows that the inclusion $\Hom^h(G,\Diff(S^1))\xrightarrow\sim\Hom^h(G,\Homeq(S^1))$ from the space of homotopy actions by diffeomorphisms into the space of homotopy actions by homotopy equivalences is a homotopy equivalence.
The following result compares strict actions and homotopy actions:

\begin{proposition}\label{nielsencircle}
For a finite group $G$, the inclusion of the space of strict actions $G\acts S^1$ into the space of homotopy actions $G\hacts S^1$ is a homotopy equivalence.
\end{proposition}

\begin{proof}
We consider specifically the inclusion of strict actions $G\acts S^1$ into homotopy actions $G\hacts S^1$ by diffeomorphisms.
It is equivalent to show that on any given circle bundle over $BG$, the space of flat connections is contractible.

To show that this space of flat connections is contractible, we introduce geometric structures into the picture.
Given a circle bundle over $BG$, we may choose a fiberwise metric of unit length, and moreover this is a contractible choice.
Similarly, given a circle bundle over $BG$ with flat connection, we may choose a fiberwise metric of unit length which is parallel with respect to the connection; this is also a contractible choice (it is equivalent to choosing a metric of length $\left|G\right|^{-1}$ on the quotient orbifold).
Note that this works well in families of strict actions due to Proposition \ref{closeconjugate}.
Hence it suffices to show that on any given circle bundle over $BG$ with fiberwise metric of unit length, the space of flat connections preserving the metric is contractible.

Since the Lie algebra $\isom(S^1)$ of the structure group $\Isom(S^1)$ is abelian, the space of metric preserving flat connections is convex, and hence is either empty or contractible.
To show that the space of flat connections is nonempty, argue as follows.
The pullback of the bundle to $EG$ is trivial (since $EG$ is contractible) and thus has a flat connection.
Averaging this flat connection (which is possible since $\isom(S^1)$ is abelian) over the action of translation by $G$ on $EG$ produces a flat connection which descends to $BG$ as desired.

An equivalent algebraic version of this averaging/descent argument is to note that the obstruction to the existence of a flat connection lies in the group $H^2(G,\isom(S^1))$, which both is a vector space over $\RR$ (since $\isom(S^1)$ is) and is annihilated by $\left|G\right|$ (since $G$ is finite).
To see that the obstruction to the existence of a flat connection lies in $H^2(G,\isom(S^1))$, we can argue as follows.
A flat connection always exists over the $1$-skeleton of $BG$, and given a flat connection over the $1$-skeleton, the obstruction to extending it to the $2$-skeleton is a $2$-cochain valued in the universal cover of $\Isom(S^1)$, which is naturally identified with its Lie algebra $\isom(S^1)$.
Consideration of the $3$-cells of $BG$ shows that this obstruction $2$-cochain is a $2$-cocycle, and modifying the given flat connection over the $1$-skeleton allows us to change this obstruction $2$-cocycle by an arbitrary $2$-coboundary.
\end{proof}

We now turn to the case of surfaces.
For surfaces $F$ which are $K(\pi,1)$ spaces (i.e.\ anything other than $S^2$ or $P^2$), the natural map $\Diff(F)\to\Homeq(F)$ is a homotopy equivalence \cite{smale,earleeells,earleschatz,gramain}, and hence the spaces of homotopy actions on $F$ by homotopy equivalences and by diffeomorphisms are homotopy equivalent.

We begin by comparing strict actions and homotopy actions on hyperbolic surfaces (a surface will be called hyperbolic iff it has negative Euler characteristic, which implies it is a $K(\pi,1)$).

\begin{proposition}\label{nielsensurfacehyperbolic}
Let $F$ be a compact hyperbolic surface-with-boundary.
The inclusion of strict actions $G\acts F$ into homotopy actions $G\hacts F$ is a homotopy equivalence.
\end{proposition}

\begin{proof}
Denote by $\Teich(F)$ the space of isotopy classes of cusped hyperbolic metrics on $F^\circ$ (equivalently, this is the space of isotopy classes of punctured conformal structures on $F^\circ$).
Note that every isotopy class is contractible (as its stabilizer inside the identity component $\Diff_0(F^\circ)$ is trivial: this holds because a biholomorphism of the unit disk is determined by its action on the boundary).

Note that a homotopy action $G\hacts F$ gives rise to a strict action $G\acts\Teich(F)$ (a bundle with fiber $F$ thus gives rise to a bundle with fiber $\Teich(F)$ with flat connection).
By Kerckhoff \cite{kerckhoff} and Wolpert \cite{wolpert}, for any homotopy action $G\hacts F$ by a finite group $G$, the fixed locus $\Teich(F)^G$ is non-empty and ``convex'' in an appropriate sense.
We do not recall the precise sense of convexity (Kerckhoff and Wolpert use different notions), rather we only note that it implies contractibility (which is all we need).

We now begin the actual argument.
Starting with a homotopy action $G\hacts F$ by diffeomorphisms (equivalently, a bundle over $BG$ with fiber $F$), we choose a point in $\Teich(F)^G$ (equivalently, a flat section of the induced bundle with fiber $\Teich(F)$); by Kerckhoff and Wolpert, this is a contractible choice.
We now upgrade this to a choice of fiberwise hyperbolic metric (this is a contractible choice as noted above: every isotopy class of hyperbolic metrics is contractible).
Now there is a unique flat connection on our bundle over $BG$ with fiber $F$ preserving this fiberwise metric.
Finally, we wish to forget this metric, leaving only the flat bundle over $BG$ with fiber $F$ (i.e.\ the strict action $G\acts F$).
Choosing a hyperbolic metric on the quotient orbifold $F/G$ is a contractible choice (this can be seen in two steps: the Teichm\"uller space is contractible, and so is every isotopy class of hyperbolic metric).
Note that this works well in families due to Proposition \ref{closeconjugate}.
\end{proof}

We now extend the above result to all $K(\pi,1)$ surfaces, using the reasoning from Proposition \ref{nielsencircle}.

\begin{proposition}\label{nielsensurfacekpione}
Let $F$ be a compact surface-with-boundary which is a $K(\pi,1)$.
The inclusion of strict actions $G\acts F$ into homotopy actions $G\hacts F$ is a homotopy equivalence.
\end{proposition}

\begin{proof}
There are five cases not covered by Proposition \ref{nielsensurfacehyperbolic}, namely $D^2$, $T^2$, $K^2$, $S^1\times I$, and $S^1\tildetimes I$.
We extend the proof to treat these cases as follows.
Instead of Teichm\"uller space, we consider the space of isotopy classes of spherical metrics (for $D^2$) and flat metrics with geodesic boundary (in the remaining cases).
These spaces are again contractible, as are the spaces of metrics in any given isotopy class.
The only difference in the proof comes when we want to find a flat connection preserving the metric.
There is now not a unique such flat connection, however as the structure groups of isometries in all cases have abelian Lie algebras, the spaces of flat connections are contractible by the argument used to prove Proposition \ref{nielsencircle}.
\end{proof}

Let us now deduce, as formal consequences, various `rel boundary' versions of the above results.

\begin{corollary}\label{surfacerelboundary}
Let $F$ be a compact surface-with-boundary which is a $K(\pi,1)$, and fix a germ of strict action of $G$ on $\Nbd\partial F$.
The inclusion of strict actions $G\acts F$ restricting to the given action on $\Nbd\partial F$ into homotopy actions $G\hacts F$ restricting to the given action on $\Nbd\partial F$ is a homotopy equivalence.
\end{corollary}

\begin{proof}
Since $\Hom^h(G,\Diff(F))\to\Hom^h(G,\Diff(\partial F)))$ is a fibration (since $\Diff(F)\to\Diff(\partial F)$ is) and $\Hom(G,\Diff(\partial F))\to\Hom^h(G,\Diff(\partial F))$ is a homotopy equivalence, it follows that the inclusion of homotopy actions $G\hacts M$ which are strict on $\partial F$ into all homotopy actions is a homotopy equivalence.
Combining this with the homotopy equivalence $\Hom(G,\Diff(F))\to\Hom^h(G,\Diff(F))$ from Proposition \ref{nielsensurfacekpione}, we conclude that the inclusion of strict actions $G\acts M$ into homotopy actions $G\hacts F$ which are strict over $\partial F$ is a homotopy equivalence.
Both sides of this homotopy equivalence are fibrations over $\Hom(G,\Diff(\partial F))$ (for the domain, this is Corollary \ref{restrictionfibration}), and hence their fibers are homotopy equivalent.
In other words, given a fixed action $G\acts\partial F$, the spaces of strict/homotopy actions on $F$ whose restriction to $\partial F$ are this given action are homotopy equivalent.
This differs from the desired result only in that it concerns agreement over $\partial F$ rather than $\Nbd\partial F$.

To conclude, it thus suffices to show that the inclusion of strict (resp.\ homotopy actions) on $F$ which agree with a given action over $\Nbd\partial F$ into strict (resp.\ homotopy) actions on $F$ which agree with the given action over $\partial F$ is a homotopy equivalence.
For the case of homotopy actions, this is trivial.
For strict actions, it suffices to describe a canonical (up to contractible choice) procedure for modifying a given strict action $\varphi$ agreeing with a fixed action $\varphi_0$ over $\partial F$ to make it agree with $\varphi_0$ over $\Nbd\partial F$ (and which furthermore does nothing if $\varphi$ already agrees with $\varphi_0$ over $\Nbd\partial F$).
Here is such a procedure.
Fix a $\varphi_0(G)$-equivariant collar $i_0:\partial F\times[0,\varepsilon)\to F$, and choose a $\varphi(G)$-equivariant collar $i:\partial F\times[0,\varepsilon)\to F$ (Corollary \ref{stricthascollar} provides a well defined up to contractible choice construction of such a collar $i$, which we may further assume coincides with $i_0$ if $\varphi=\varphi_0$ over $\Nbd\partial F$).
Now the space of collars is contractible, so we may deform $i$ to $i_0$.
We may extend this deformation of collars to a family of diffeomorphisms of $F$ fixed on $\partial F$, and thus (by conjugating) to a deformation of strict actions $\varphi$.
After this deformation, the equivariant collars coincide, and hence so do the actions over $\Nbd\partial F$.
\end{proof}

\begin{corollary}\label{relsurfaceboundary}
Let $\varphi:G\hacts M$ be a homotopy action by homotopy equivalences on a three-manifold-with-boundary $M$ all of whose boundary components are closed surfaces other than $S^2$ or $P^2$.
If $\varphi$ is strict over $\Nbd\partial M$ and $\varphi$ is homotopic to a strict action, then this homotopy may be taken to be constant over $\Nbd\partial M$.
\end{corollary}

\begin{proof}
Let a homotopy $\varphi_t$ ($t\in[0,1]$) from $\varphi=\varphi_0$ to a strict action $\varphi_1$ be given.
By Proposition \ref{nielsensurfacekpione}, the restriction of $\varphi_t$ to $\partial M$ may be deformed relative $t=0,1$ to stay within strict actions.
Extend this deformation to $M$ (possible since $\Hom^h(G,\Homeq(M))\to\Hom^h(G,\Homeq(\partial M))$ is a fibration since $\Homeq(M)\to\Homeq(\partial M)$ is), so we may assume without loss of generality that the restriction of $\varphi_t$ to $\partial M$ is strict for all $t$.

Using the local triviality of deformations of strict actions (Proposition \ref{closeconjugate}), we see that there exists a family of diffeomorphisms $\rho_t$ of $\partial M$ starting at $\rho_0=\id$ such that the conjugated family $\rho_t\varphi_t\rho_t^{-1}$ is constant (independent of $t$).
Extending this family $\rho_t$ to diffeomorphisms of all of $M$ and replacing $\varphi_t$ with $\rho_t^{-1}\varphi_t\rho_t$, we may assume that $\varphi_t$ is itself constant over $\partial M$.

Since $\varphi_0$ and $\varphi_1$ are both strict over near $\partial M$, there are equivariant boundary collars $i_0,i_1:\partial M\times[0,\varepsilon)\hookrightarrow M$ (Corollary \ref{stricthascollar}).
By again conjugating $\varphi_t$ by an appropriate family of diffeomorphisms of $M$ acting as the identity on $\partial M$, we may assume that these boundary collars coincide, and hence that $\varphi_0=\varphi_1$ over $\Nbd\partial M$.
Finally, let $\hat M$ denote the result of using this boundary collar to attach $\partial M\times[0,1]$ to $M$, and let $\hat\varphi_t$ denote the extension of $\varphi_t$ to $\hat M$ defined by acting on $\partial M\times[0,1]$ via the restriction of $\varphi_t$ to $\partial M$ and the trivial action on $[0,1]$.
Now this is the desired homotopy (note that $(M,\varphi_0)$ and $(\hat M,\hat\varphi_0)$ are diffeomorphic).
\end{proof}

\subsection{Nielsen realization for Seifert fibered three-manifolds}

The Nielsen realization problem for Seifert fibered three-manifolds is well studied, see Heil--Tollefson \cite{heiltollefsonI}, Zimmermann \cite{zimmermannseifert}, and Meeks--Scott \cite{meeksscott}.
We give here a brief derivation of the version of these results which we will need.

Let us recall the definition of a Seifert fibration and some of its basic properties; more details may be found in Scott \cite[\S 3]{scottgeometries}.
A Seifert fibration on an orientable three-manifold $M$ is a one-dimensional foliation $\F$ such that $(M,\F)$ has an open cover by local models of the form $(D^2\times S^1,TS^1)/(\ZZ/n)$ where $1\in\ZZ/n$ acts by rotation by $2\pi/n$ on $S^1$ and by rotation by $2\pi k/n$ on $D^2$ where $k$ is relatively prime to $n$.
When $M$ may have boundary, an additional local model $(S^1\times\RR\times\RR_{\geq 0},TS^1)$ is also allowed.
(There are yet more local models relevant if $M$ is non-orientable or is an orbifold, however we will not encounter these cases in this paper.)
The leaves of $\F$ are called the fibers of the fibration.
The central fiber of the local model $(D^2\times S^1,TS^1)/(\ZZ/n)$ will be called a multiple fiber of multiplicity $n$; all other fibers are called regular fibers.

For a Seifert fibration $(M,\F)$, the holonomy groupoid of $(M,\F)$ presents a(n effective) surface orbifold $B$, and the resulting projection $M\to B$ is also referred to as a Seifert fibration (it determines $\F$ as the kernel of its derivative); we could in fact simply define a Seifert fibration as a circle bundle over a surface orbifold.
The orbifold points with $\ZZ/n$ isotropy on $B$ correspond to the multiple fibers of multiplicity $n$.
Given any orbifold covering $B'\to B$, the Seifert fibration $M\to B$ pulls back to a Seifert fibration $M'\to B'$, with $M'\to M$ being a covering space.
If $B$ is a $K(\pi,1)$ orbifold (in the sense that its orbifold universal cover $\tilde B$ is a contractible manifold), then the pullback of $M\to B$ to $\tilde B$ is a necessarily trivial circle bundle, so we see that the universal cover of $M$ is $\tilde B\times\RR$.
Thus $M$ is a $K(\pi,1)$, and there is a short exact sequence
\begin{equation}
1\to\pi_1(S^1)\to\pi_1(M)\to\pi_1(B)\to 1
\end{equation}
where $\pi_1(B)$ denotes the orbifold fundamental group (compare \cite[Lemmas 3.1 and 3.2]{scottgeometries}).
The subgroup $\ZZ=\pi_1(S^1)\subseteq\pi_1(M)$ will be called the fiber subgroup.

\begin{proposition}\label{nielsenseifert}
Let $M$ be a compact orientable three-manifold-with-boundary which admits a Seifert fibration $M\to B$ over a hyperbolic base orbifold-with-boundary $B$.
Every homotopy action by homotopy equivalences of a finite group on $M$ is homotopic to a strict action.
\end{proposition}

\begin{proof}
We begin with some preliminary observations about the fundamental group of $M$.
First, let us argue that the fiber subgroup of $\pi_1(M)$ is characterized intrinsically as those elements $x\in\pi_1(M)$ for which $gxg^{-1}\in\{x,x^{-1}\}$ for all $g\in\pi_1(M)$.
It is easy to see that elements of the fiber subgroup satisfy this property, so the point is to prove the converse.
If $x\in\pi_1(M)$ satisfies $gxg^{-1}\in\{x,x^{-1}\}$ for all $g\in\pi_1(M)$, then we conclude the same is true for the image of $x$ in $\pi_1(B)$.
On the other hand, using the dynamical classification of elements of $\Isom(\HH^2)$ and the fact that the limit set of $\pi_1(B)\subseteq\Isom(\HH^2)$ is the entire unit circle, it is easy to conjugate any nontrivial element of $\pi_1(B)$ to become distinct from itself and its inverse.
From this intrinsic characterization of the fiber subgroup, it follows immediately that any self homotopy equivalence of $M$ preserves the fiber subgroup (and hence acts on it as either plus or minus the identity).
Next, let us give an intrinsic characterization of the multiple fibers.
Obviously if $x$ is the class of a multiple fiber of multiplicity $n$, then $x^n$ is the fiber class.
Conversely, suppose that $x\in\pi_1(M)$ is such that $x^n$ is the fiber class.
The image of $x$ in $\pi_1(B)$ is thus a nontrivial torsion element.
Nontrivial torsion elements in $\pi_1(B)\subseteq\Isom(\HH^2)$ fix a unique point of $\HH^2$; hence the image of $x$ in $\pi_1(B)$ is the conjugacy class of a `loop' concentrated at a unique orbifold point of $B$.
It follows that the conjugacy class of $x$ itself is a (necessarily unique) multiple of a unique multiple fiber.

We now begin the process of deforming our given homotopy action $G\hacts M$ to make it strict.
Our first step is to deform it to make it land in the submonoid $\Homeq(M,\F)\subseteq\Homeq(M)$ consisting of those self homotopy equivalences of $M$ which send $\F$ into itself (meaning tangent vectors in $\F$ push forward to tangent vectors in $\F$).
To do this, it suffices to show that the inclusion $\Homeq(M,\F)\subseteq\Homeq(M)$ is a homotopy equivalence.
Fix a triangulation of the base orbifold $B$ whose vertices include all the orbifold points.
We may now build $\Homeq(M,\F)$ from $\Homeq(M)$ in steps, imposing the constraint (of sending $\F$ into itself) first over a neighborhood of the inverse image of the $0$-skeleton, then of the $1$-skeleton, and then everywhere.
It thus suffices to show that for any (homotopy class of) self homotopy equivalence $f$ of $M$ and any fiber or multiple fiber class $\alpha$, the inclusion
\begin{equation}\label{tangentinclusion}
\Maps((S^1,TS^1),(M,\F))_{f(\alpha)}\hookrightarrow\Maps(S^1,M)_{f(\alpha)}
\end{equation}
of maps $S^1\to M$ tangent to $\F$ into all maps $S^1\to M$ (both in the homotopy class $f(\alpha)$) is a homotopy equivalence.
Note that $f(\alpha)$ is itself either the fiber class or the class of a unique multiple fiber, by our discussion in the previous paragraph.
If $f(\alpha)$ is the fiber class, then both sides of \eqref{tangentinclusion} fiber over $M$ (by evaluation at a basepoint of $S^1$) with fibers (specific components of) $\Omega S^1$ and $\Omega M$, respectively, which are both contractible.
If $f(\alpha)$ is the class of a multiple fiber, then by our discussion in the previous paragraph, the domain of \eqref{tangentinclusion} consists solely of maps into that particular multiple fiber (hence is homotopy equivalent to $S^1$), and the target has homotopy type calculated by Lemma \ref{mapskpione}, also giving $S^1$, since the centralizer of the multiple fiber is only the infinite cyclic group it generates.

We have thus reduced ourselves to a homotopy homomorphism $G\hto\Homeq(M,\F)$.
Such a homotopy homomorphism induces a homotopy action $G\hacts B$ on the base orbifold $B$ via the forgetful map $\Homeq(M,\F)\to\Homeq(B)$.
(Concretely, a homotopy equivalence $B\to B$ is an isomorphism $\alpha:\pi_1(B)\to\pi_1(B)$ and an $\alpha$-equivariant map $\tilde B\to\tilde B$, modulo simultaneous conjugation by $\pi_1(B)$.)
Now applying the orbifold version of Proposition \ref{nielsensurfacehyperbolic}, we may deform this homotopy action to a strict action $G\acts B$.
We may now lift this deformation to a deformation of homotopy homomorphisms $G\hto\Homeq(M,\F)$ using the fact that $M\to B$ is a submersion.

We now have a homotopy homomorphism $\varphi:G\hto\Homeq(M,\F)$ inducing a strict action $\varphi_B:G\acts B$.
Our final step is now to deform $\varphi$ relative $\varphi_B$ to become strict.
To do this, we follow the argument of Proposition \ref{nielsencircle} (which we cannot simply quote directly, since our current circumstances essentially require an \emph{equivariant} version of Proposition \ref{nielsencircle}).
Begin with an arbitrary Riemannian metric on $\F$, and note that the lengths of fibers (multiple fibers counted with multiplicity) define a smooth function on $M$, so dividing by its square gives a metric $g$ on $\F$ for which all fibers have unit length.
Now $\Homeq(M,\F)$ deformation retracts (relative the forgetful map to $\Homeq(B)$) onto the subspace $\Homeq(M,\F,g)$ of maps which preserve the metric on $\F$; hence we may deform $\varphi$ relative $\varphi_B$ to land in $\Homeq(M,\F,g)$.
We thus have a homotopy homomorphism $\varphi:G\hto\Diff(M,\F,g)$ lifting our fixed strict action $\varphi_B:G\acts B$.
We define a (strict) homomorphism $\bar\varphi:G\to\Diff(M,\F,g)$ by the formula
\begin{equation}\label{avgdefn}
\bar\varphi(g):=\frac 1{|G|}\sum_{a\in G}\varphi(ga)\varphi(a)^{-1}
\end{equation}
which we now explain (we use the shorthand $\varphi=\varphi^0$).
Note that the kernel $A:=\ker(\Diff(M,\F,g)\to\Diff(B))$ is abelian.
Each term $\varphi(ga)\varphi(a)^{-1}\in\Diff(M,\F,g)$ lies in the fiber of $\Diff(M,\F,g)\to\Diff(B)$ over $\varphi_B(g)$, which is an $A$-torsor (principal homogeneous space for $A$).
Now $\varphi^1(g,a)\varphi(a)^{-1}$ provides a path within this $A$-torsor between $\varphi(ga)\varphi(a)^{-1}$ and $\varphi(g)$.
The data of these paths is sufficient to define the average in this $A$-torsor of $\varphi(ga)\varphi(a)^{-1}$ over $a\in G$, and this is the meaning of the right side of \eqref{avgdefn}.
Note thus that \eqref{avgdefn} combines additive and multiplicative notation for the same group operation.
To check that $\bar\varphi(g)$ is a group homomorphism, we calculate
\begin{equation}
\bar\varphi(g)\bar\varphi(h):=\frac 1{|G|}\sum_{b\in G}\varphi(ghb)\varphi(hb)^{-1}\varphi(hb)\varphi(b)^{-1}
\end{equation}
where the average now takes place in the $A$-torsor over $\varphi_B(g)\varphi_B(h)=\varphi_B(gh)$ using the paths from each term $\varphi(ghb)\varphi(b)^{-1}$ to $\varphi(g)\varphi(h)$ given by concatenation of $\varphi^1(g,hb)\varphi(b)^{-1}$ and $\varphi(g)\varphi^1(h,b)\varphi(b)^{-1}$.
Now $\varphi^2(g,h,b)$ provides a homotopy between these paths and the paths $\varphi^1(gh,b)\varphi(b)^{-1}$ to $\varphi(gh)$ and then $\varphi^1(g,h)$ to $\varphi(g)\varphi(h)$.
We may dispense with the second path since it is independent of $b$, so we see that the right side above coincides with $\bar\varphi(gh)$.
We now deform our homotopy action $\varphi:G\hto\Homeq(M,\F)$ relative $\varphi_B$ by homotoping $\varphi^0$ to $\bar\varphi$.

At this point, we have a homotopy homomorphism $\varphi:G\hto\Homeq(M,\F)$ which lifts a strict action $\varphi_B:G\acts B$ and whose first component $\varphi^0$ is a group homomorphism (note that this is weaker than $\varphi$ being strict, which also entails all higher $\varphi^k$ being constant).
Now $\varphi^1(x,y)$ is a path from $\varphi^0(x)\varphi^0(y)$ to $\varphi^0(xy)$; since these are equal, $\varphi^1(x,y)$ determines an element of $\pi_1A$ where $A=\ker(\Diff(M,\F,g)\to\Diff(B))$ is as before.
The existence of $\varphi^2$ means that this function $G\times G\to\pi_1A$ is a $2$-cocycle for the action of $G$ on $\pi_1A$ via $\varphi^0$.
Now $H^2(G,\pi_1A\otimes_\ZZ\RR)=0$ since $G$ is finite, so this $2$-cocycle is the coboundary of a $1$-cochain valued in $\pi_1A\otimes_\ZZ\RR$.
Compose this $1$-cochain with the exponential map $\pi_1A\otimes_\ZZ\RR\to A$, and deform $\varphi$ by multiplying $\varphi^0$ by this $1$-cochain; note that $\varphi^0$ remains a group homomorphism since the coboundary of this $1$-cochain lies in $\pi_1A$ which is annihilated by the exponential map.
After this deformation, we may now null homotope $\varphi^1(x,y)$ rel its endpoints and rel $\varphi_B$.
Finally, we may null homotope $\varphi^2$ rel boundary and rel $\varphi_B$ since $\pi_2A=0$, and similarly for all higher components of $\varphi$ inductively.
\end{proof}

\begin{lemma}\label{nielsenttwoi}
Every homotopy action by homotopy equivalences of a finite group on $T^2\times I$ is homotopic to a strict action.
\end{lemma}

\begin{proof}
The map $\Homeq(T^2)\to\Homeq(T^2\times I)$ is a homotopy equivalence, hence every homotopy action on $T^2\times I$ is homotopic to one acting only on the $T^2$ coordinate.
Now apply Proposition \ref{nielsensurfacekpione}.
\end{proof}

\begin{lemma}\label{nielsendtwosone}
Every homotopy action by homotopy equivalences of a finite group on $D^2\times S^1$ is homotopic to a strict action.
\end{lemma}

\begin{proof}
We first claim that the map $\Homeq(D^2\times S^1)\to\Homeq(\partial D^2\times S^1)$ is a homotopy equivalence onto the components in its image (which are precisely those mapping classes of $\partial D^2\times S^1$ which preserve the free homotopy class of loops $\partial D^2\times\{*\}$).
Indeed, given a self homotopy equivalence of $\partial D^2\times S^1$, the choice of an extension to $D^2\times\{*\}$ (mapping to $D^2\times S^1$) is contractible (since the second based loop space $\Omega^2(D^2\times S^1)$ is contractible), as is the subsequent choice of an extension to the remaining $3$-cell (since $\Omega^3(D^2\times S^1)$ is contractible).

It follows that a homotopy action $G\hacts D^2\times S^1$ is the same (up to homotopy equivalence) as a homotopy action $G\hacts\partial D^2\times S^1$ which preserves the free homotopy class of loops $\partial D^2\times\{*\}$.
The desired result may thus be stated alternatively as: every homotopy action by homotopy equivalences of a finite group on $\partial D^2\times S^1$ preserving the free homotopy class of loops $\partial D^2\times\{*\}$ is homotopic to the restriction of a strict action on $D^2\times S^1$.

This equivalent statement now follows easily from Proposition \ref{nielsensurfacekpione}.
Indeed, Proposition \ref{nielsensurfacekpione} implies that any given homotopy action $G\hacts\partial D^2\times S^1$ may be deformed to a strict action by Proposition \ref{nielsensurfacekpione}.
Such a strict action $G\acts\partial D^2\times S^1$, either by its construction from the proof of Proposition \ref{nielsensurfacekpione} or by \cite[Theorem 2.4]{scottgeometries}, preserves some flat metric, hence in particular preserves the affine structure on $\partial D^2\times S^1$, namely it lands inside $(\RR^2/\ZZ^2)\rtimes\GL_2(\ZZ)$.
Since it preserves the free homotopy class of loops $\partial D^2\times\{*\}$, it in fact lands inside the subgroup $(\RR^2/\ZZ^2)\rtimes\left(\begin{smallmatrix}*&*\\0&*\end{smallmatrix}\right)\leq(\RR^2/\ZZ^2)\rtimes\GL_2(\ZZ)$, whose action on $\partial D^2\times S^1$ naturally extends to $D^2\times S^1$.
\end{proof}

\begin{lemma}\label{seifertinthreespace}
Let $M\to B$ be a Seifert fibration of a compact three-manifold-with-boundary.
If $M$ admits an embedding into $\RR^3$, then either $B$ is hyperbolic or $M=D^2\times S^1$ or $T^2\times I$.
\end{lemma}

(See also the classification in Budney \cite[Proposition 4]{budneyjsj}.)

\begin{proof}
Clearly $\partial M\ne\varnothing$, so the base orbifold $B$ must have nonempty boundary.
There is thus only a small list of non-hyperbolic base orbifolds for us to consider.
If the base is $D^2$ with $\leq 1$ orbifold points, then the total space is $S^1\times D^2$.
If the base is $D^2$ with two orbifold points with $\ZZ/2$ isotropy, then the total space has an embedded Klein bottle (the inverse image of an arc between the two orbifold points) and thus cannot embed into $\RR^3$.
If the base is an annulus $S^1\times I$, then the total space is either $T^2\times I$ or non-orientable and thus cannot embed into $\RR^3$.
If the base is a M\"obius strip $S^1\tildetimes I$, then the total space is either non-orientable or contains an embedded Klein bottle and thus cannot embed into $\RR^3$.
\end{proof}

\subsection{Nielsen realization for hyperbolic three-manifolds}

A Nielsen realization type result for hyperbolic three-manifolds follows from the deep results of Ahlfors, Bers, Kra, Marden, Maskit, and Mostow, as we now recall (for detailed discussion, see also \cite{matsuzakitaniguchi,kapovich,mardenbook}).

Given a group $\Gamma$, denote by $X(\Gamma)$ the set of representations $\rho:\Gamma\to\PGL_2\CC=\Isom^+(\HH^3)$ up to conjugation.
We can regard $X(\Gamma)$ as a groupoid, in which an object is a representation $\rho:\Gamma\to\PGL_2\CC$ and an isomorphism $\rho\to\rho'$ is an element $\gamma\in\PGL_2\CC$ satisfying $\gamma\rho\gamma^{-1}=\rho'$.
This latter perspective leads naturally to the observation that $X(\Gamma)$ makes sense more generally for any groupoid $\Gamma$, namely it is the groupoid of functors from $\Gamma$ to the groupoid $B\PGL_2\CC$ with a single object whose automorphism group is $\PGL_2\CC$.
Later, we will be interested specifically in the case $\Gamma=\pi_1(M)$ is the fundamental groupoid of a manifold $M$.
By speaking of groupoids instead of groups, we can avoid choosing a basepoint on $M$ or assuming that $M$ is connected.
Even though our `official' perspective is to work with groupoids, we will sometimes slip into the more familiar language of groups in the discussion which follows.

Embedding $X(\Gamma)$ into $(\PGL_2\CC)^r/\PGL_2\CC$ via $\rho\mapsto(\rho(\gamma_1),\ldots,\rho(\gamma_r))$ for generators $\gamma_1,\ldots,\gamma_r\in\Gamma$ gives $X(\Gamma)$ the structure of a (possibly singular and possibly non-separated) complex analytic stack.

Given $\rho\in X(\Gamma)$, we can form $M_\rho:=\colim_\Gamma\rho^\ast\HH^3$.
If $\Gamma$ is a group, then $M_\rho=\HH^3/\Gamma_\rho$ is simply the quotient of $\HH^3$ by $\Gamma$ acting via $\rho:\Gamma\to\PGL_2\CC\acts\HH^3$.
Given $\rho_1\in X(\Gamma_1)$, $\rho_2\in X(\Gamma_2)$, a homomorphism $\alpha:\Gamma_1\to\Gamma_2$, and an isomorphism $\rho_1\xrightarrow\sim\rho_2\circ\alpha$, we obtain a map $M_{\rho_1}\to M_{\rho_2}$ which is a local isometry.

The quotient $M_\rho$ is separated iff the action of $\Gamma$ on $\HH^3$ via $\rho$ is proper, meaning $\Gamma\times\HH^3\to\HH^3\times\HH^3$ is proper.
Since the action $\PGL_2\CC\acts\HH^3$ is proper, this is equivalent to $\rho:\Gamma\to\PGL_2\CC$ being proper (i.e.\ finite kernel and discrete image).
A representation $\rho:\Gamma\to\PGL_2\CC$ which is proper is called a \emph{Kleinian group}, and the set of such $\rho\in X(\Gamma)$ is denoted $H(\Gamma)\subseteq X(\Gamma)$.
Thus for $\rho\in H(\Gamma)$, we have an orbifold $M_\rho$ which comes with a canonical equivalence of groupoids $\pi_1(M_\rho)=\Gamma$.

The action of $\PGL_2\CC$ on $\HH^3$ extends to an action on the ideal boundary $\partial\HH^3=S^2_\infty$ the Riemann sphere by biholomorphisms, and in fact $\PGL_2\CC=\Con^+(S^2_\infty)$ is precisely the set of all orientation preserving conformal automorphisms of $S^2_\infty$.
The action of $\Gamma$ on $S^2_\infty$ induces a decomposition $S^2_\infty=\Omega_\rho\cup\Lambda_\rho$ into the open \emph{set of discontinuity} $\Omega_\rho$ and its complement the closed \emph{limit set} $\Lambda_\rho$.
For $\rho\in H(\Gamma)$, the orbifold $M_\rho$ admits a natural partial compactification $\overline M_\rho$ defined as the quotient of $\HH^3\cup\Omega_\rho$ by $\Gamma$.

A Kleinian group $\rho\in H(\Gamma)$ is called \emph{geometrically finite} iff the associated action on $\HH^3$ has a finite sided polytope as fundamental domain.
When $\left|\Lambda_\rho\right|>1$, this condition is equivalent to the $\varepsilon$-neighborhood of the convex core of $M_\rho$ having finite volume for some (equivalently every) $\varepsilon>0$.
Denote by $GF(\Gamma)\subseteq H(\Gamma)$ the collection of $\rho\in H(\Gamma)$ which are geometrically finite.

For $\rho\in GF(\Gamma)$, the manifold $\overline M_\rho$ has a natural compactification $\overline{\overline M}_\rho$ which is a compact three-manifold-with-boundary.
We have $\overline M_\rho=\overline{\overline M}_\rho\setminus P$ where $P\subseteq\partial\overline{\overline M}_\rho$ is a codimension zero submanifold-with-boundary called a \emph{pared structure} (consisting of tori and annuli satisfying some axioms, see Morgan \cite[Definition 4.8]{morgan} or Canary--McCullough \cite[\S 5]{canarymccullough} or Kapovich \cite[\S 1.5]{kapovich}), marking the elements of $\Gamma$ which $\rho$ sends to parabolics.

\begin{theorem}\label{nielsenhyperbolic}
Let $M\ne B^3$ be a compact three-manifold-with-boundary whose interior admits a geometrically finite hyperbolic metric with pared structure $P=(\partial M)_{\chi=0}$.
Every homotopy action by homotopy equivalences of a finite group on $M$ lifts to a strict action.
\end{theorem}

\begin{proof}
By assumption, $M=\overline{\overline M}_\rho$ for some geometrically finite $\rho:\Gamma\to\PGL_2\CC$ with pared structure $(\partial M)_{\chi=0}$.
Note that all boundary components of $M$ have non-positive Euler characteristic (any $S^2$ inside $M$ lifts to $\HH^3$, where it bounds a $B^3$, and hence also bounds a $B^3$ in $M$).
A choice of $\rho$ gives rise to an isotopy class of conformal structure $\xi_\rho\in\Teich(\partial M)$ and thus also to $\xi_\rho^-\in\Teich((\partial M)_{\chi<0})$.
By Bers \cite{bers} (see also Kra \cite{kra} and Maskit \cite{maskit}), we may deform $\rho$ (by \emph{quasi-conformal conjugacy}) so as to induce any arbitrary $\xi^-\in\Teich((\partial M)_{\chi<0})$.
By Kerckhoff \cite{kerckhoff} (or Wolpert \cite{wolpert}), there exists such a $\xi^-\in\Teich((\partial M)_{\chi<0})$ which is fixed by the action of $G$.
Fix any $\rho$ whose induced $\xi_\rho^-$ is such a fixed point, and consider $M$ equipped with the hyperbolic metric associated to $\rho$.
By Mostow/Prasad/Marden rigidity \cite{mostow,prasad,marden}, every element of $\pi_0\Homeq(M\setminus(\partial M)_{\chi=0})$ preserving $\xi_\rho^-$ is represented by a unique isometry of $M$.
In particular, this implies that there is a strict action of $G$ on $M$ which coincides on $\pi_0\Homeq(M\setminus(\partial M)_{\chi=0})$ with our given action.
Finally, note that the maps $\Homeq(M)\to\Homeq(M\setminus(\partial M)_{\chi=0})\to\pi_0\Homeq(M\setminus(\partial M)_{\chi=0})$ are all homotopy equivalences.
\end{proof}

\subsection{Some three-manifold topology}

We recall some well known fundamental results in three-manifold topology.

\begin{definition}
A three-manifold-with-boundary $M$ is called \emph{irreducible} iff every embedded $S^2$ inside $M$ is the boundary of an embedded $B^3$.
It is called \emph{$P^2$-irreducible} iff it is irreducible and there exists no two-sided embedding $P^2\hookrightarrow M$.
By the sphere theorem \cite{papakyriakopoulos,stallings,hempel}, if $M$ is $P^2$-irreducible then $\pi_2(M)=0$.
\end{definition}

\begin{lemma}\label{mfldkpione}
A $P^2$-irreducible three-manifold which is either non-compact or has infinite fundamental group is a $K(\pi,1)$.
\end{lemma}

\begin{proof}
Since $\pi_2=0$, to check that the universal cover is contractible, it is enough (by Hurewicz) to show that its $H_3$ vanishes, which follows since it is non-compact.
\end{proof}

\begin{definition}
A compact properly embedded surface-with-boundary $(F,\partial F)\hookrightarrow(M,\partial M)$ inside a $P^2$-irreducible three-manifold-with-boundary is called \emph{incompressible} iff every properly embedded disk $(D^2,\partial D^2)\hookrightarrow(M,F)$ (disjoint from $F$ except along its boundary) is parallel to an embedded disk inside $F$ and no component of $F$ is $S^2$.
By the loop theorem \cite{papakyriakopoulos,stallings,hempel}, a two-sided surface is incompressible iff $\pi_1(F)\to\pi_1(M)$ is injective.
\end{definition}

An innermost disk argument shows that a surface is incompressible iff each of its components is incompressible.

\begin{theorem}[Waldhausen \cite{waldhausen,heilirreducible,hatcherpirreducible,waldhausensurvey}]\label{homotopydiffeomanifolds}
Let $M$ and $N$ be compact connected $P^2$-irreducible three-manifolds-with-boundary, each of which contains a two-sided incompressible surface.
Every homotopy equivalence of pairs $(M,\partial M)\to(N,\partial N)$ is homotopic through maps of pairs to a diffeomorphism.
\qed
\end{theorem}

Note that every compact connected $P^2$-irreducible three-manifold-with-boundary with nonzero $H^1$ (which is implied if the boundary is non-empty) contains a two-sided incompressible surface (take a maximal compression of a co-oriented surface representing the Poincar\'e dual of a nonzero element of $H^1$).

\subsection{Nielsen realization for irreducible three-manifolds embedding into \texorpdfstring{$\RR^3$}{R\^{}3}}

We now combine the results of the previous two subsections using the JSJ decomposition which we now recall.

\begin{definition}
An orientable three-manifold-with-boundary $M$ is called \emph{atoroidal} iff every incompressible $T^2\hookrightarrow M$ is boundary parallel.
\end{definition}

\begin{theorem}[{JSJ Decomposition \cite{jacoshalen,johannson,neumannswarup} \cite[\S 2]{canarymccullough}}]\label{jsjdecomposition}
Let $M$ be a compact, orientable, irreducible, three-manifold-with-boundary.
There exists a unique up to isotopy minimal disjoint union of incompressible tori $\underline T\subseteq M$ such that every component of $M\setminus N^\circ\underline T$ is either Seifert fibered or atoroidal.
\qed
\end{theorem}

\begin{theorem}\label{nielsenall}
Let $M\ne B^3$ be a compact irreducible three-manifold-with-boundary which embeds into $\RR^3$.
Every homotopy action by homotopy equivalences of a finite group on $M$ is homotopic to a strict action.
\end{theorem}

\begin{proof}
Denote by $\varphi$ the given homotopy action.
By Theorem \ref{homotopydiffeomanifolds} of Waldhausen, we may deform $\varphi^0$ so that it lands in diffeomorphisms.
This deformation of $\varphi^0$ may be lifted to a deformation of $\varphi$; indeed any deformation of $\varphi^{k-1}$ lifts to a deformation of $\varphi^k$ since the boundary of $G^{k+1}\times[0,1]^k$ is collared.

Let $\underline T\subseteq M$ be a JSJ decomposition as in Theorem \ref{jsjdecomposition}.
As the isotopy class of $\underline T$ is unique, we conclude that $\varphi^0(g)(\underline T)$ is isotopic to $\underline T$.
We may thus further deform $\varphi$ so that $\varphi^0$ lands in diffeomorphisms preserving $\underline T$.

Next, let us deform $\varphi$ so that it (i.e.\ all $\varphi^k$) maps $\underline T$ to itself.
This holds already for $\varphi^0$, and we proceed by induction on $k\geq 1$.
For the inductive step, it suffices to know that $\Homeq(\underline T)\to\Maps(\underline T,M)$ is a homotopy equivalence onto the components in its image.
By Lemma \ref{mapskpione}, this is equivalent to knowing that $N_{\pi_1(M)}(\pi_1(T))=\pi_1(T)$ for every component $T$ of $\underline T$, which holds by \cite{heilnormalizers,heilincompressibleII} (a two-sided incompressible surface in a $P^2$-irreducible three-manifold-with-boundary has nontrivial normalizer iff it is the fiber of a fibration over $S^1$ or the boundary of a regular neighborhood of a one-sided surface).
Now that we have deformed $\varphi$ so that it stabilizes $\underline T$, we may apply Proposition \ref{nielsensurfacekpione} to further deform $\varphi$ so that its restriction to $\underline T$ is a strict action (note that we may perform this deformation preserving the property that $\varphi^0$ lands in diffeomorphisms stabilizing $\underline T$).
In fact, we may now even deform $\varphi$ so its restriction to a neighborhood of $\underline T$ is strict.

Finally, let us deform $\varphi$ so that it preserves the partition into components of $M\setminus\underline T$.
This holds already for $\varphi^0$, and we proceed by induction on $k\geq 1$.
We can simply deal with each component $N\subseteq M$ (compact manifold-with-boundary) separately, and it suffices to show that
\begin{equation}
\Maps((N,\partial N),(N,\partial N))\hookrightarrow\Maps((N,\partial N),(M,\partial N))
\end{equation}
is a homotopy equivalence onto the components in its image.
To analyze this inclusion, begin with the equality $\Maps((\partial N,\partial N),(N,\partial N))=\Maps((\partial N,\partial N),(M,\partial N))$ and add $k$-cells to $\partial N$ one by one to build $N$ and thus produce the above inclusion of interest.
The effect of adding a $k$-cell is that both sides get replaced by the total spaces of fibrations over them with fiber either empty or $\Omega^kN$ and $\Omega^kM$, respectively.
We may assume $k\geq 1$ since $N$ is connected and $\partial N\ne\varnothing$ (the case $\partial N=\varnothing$ only happens when $N=M$ in which case there is nothing to prove).
Now both $N$ and $M$ are $K(\pi,1)$ spaces (Lemma \ref{mfldkpione}), so for $k\geq 2$ both $\Omega^kN$ and $\Omega^kM$ are contractible, and for $k=1$ they are homotopy equivalent to $\pi_1(N)$ and $\pi_1(M)$, respectively.
It is thus enough to know that $\pi_1(N)\to\pi_1(M)$ is injective, which holds since $\partial N$ is incompressible.

Now we have deformed $\varphi$ so that it restricts to a strict action on a neighborhood of $\underline T$ and respects the partition of $M$ into components of $M\setminus\underline T$.
The resulting action on the pieces of this partition is again by homotopy equivalences.
Each of the pieces is either atoroidal or Seifert fibered.
The atoroidal pieces are hyperbolic by Thurston \cite{morgan}.
The Seifert fibered pieces all have hyperbolic base orbifold by Lemma \ref{seifertinthreespace} (there can be no $T^2\times I$ or $D^2\times S^1$ pieces unless they are the entire $M$, cases which are covered by Lemmas \ref{nielsenttwoi} and \ref{nielsendtwosone}).
Hence we may conclude by applying Theorem \ref{nielsenhyperbolic} and Proposition \ref{nielsenseifert}, as augmented by Corollary \ref{relsurfaceboundary} to be `rel boundary'.
\end{proof}

\section{A lattice of codimension zero submanifolds}\label{sectionlattice}

This section defines for certain three-manifolds a lattice of codimension zero submanifolds with incompressible boundary, generalizing the setup of \cite[\S 2]{pardonhilbertsmith}.

\subsection{Inside a surface}\label{sectionlatticesurface}

We begin with a discussion of the analogous lattice in one lower dimension, namely for surfaces, where everything is essentially elementary.

Let $F$ be a surface (without boundary, possibly non-compact).
We denote by $\sL(F)$ the set of isotopy classes of codimension zero submanifolds-with-boundary $A\subseteq F$ for which $\partial A\subseteq F$ is a compact multi-curve, such that neither $A$ nor $A^\complement:=F\setminus A^\circ$ have any components diffeomorphic to $D^2$, $S^1\times I$, or $S^1\tildetimes I$.
This implies that $\partial A$ consists of pairwise non-isotopic essential curves on $F$.
The isotopy class of any such $A\subseteq F$ is contractible.

There is a partial order on $\sL(F)$ by inclusion.
Namely, $\A\leq\A'$ iff there are representatives $A,A'\subseteq F$ of $\A$ and $\A'$ with $A\subseteq A'$.
Obviously $A\mapsto A^\complement$ is an order reversing involution of $\sL(F)$.

Equipped with this partial order, $\sL(F)$ is a lattice, namely every finite subset $S\subseteq\sL(F)$ has a least upper bound.
To see this, fix a hyperbolic metric on $F$ with no parabolics, so every isotopy class of simple closed curve has a unique geodesic representative, which is length minimizing.
Now every isotopy class $\A\in\sL(F)$ has a unique representative $A\subseteq F$ whose boundary consists of a disjoint union of length minimizing geodesics.
These representatives simultaneously realize all order relations, in the sense that for such $A,A'\subseteq F$, if $[A]\leq[A']$ then $A\subseteq A'$.
This implies the lattice property as follows.
For finite $S\subseteq\sL(F)$, consider the unique representatives $A_s\subseteq F$ whose boundaries $\partial A_s$ are disjoint unions of length minimizing geodesics.
The union $\bigcup_{s\in S}A_s\subseteq F$ will not have any $D^2$, $S^1\times I$, or $S^1\tildetimes I$ components, however its complement may have such.
Adding in these disallowed components produces the desired least upper bound, due to the fact that any $\A'\in\sL(F)$ with $\A'\geq[A_s]$ for all $s\in S$ has a representative $A'$ with $A'\supseteq A_s$ for all $s\in S$.

\subsection{Inside a three-manifold}\label{sectionlatticethreemanifold}

Let $M$ be a $P^2$-irreducible three-manifold-with-boundary, and let $A\subseteq\partial M$ be a codimension zero submanifold-with-boundary representing an element of $\sL(\partial M)$.
We will define a lattice $\sL(M;A)$.

Throughout this subsection, the notation $B\subseteq M$ (or its decorations such as $B'$, $B_1$, $\bar B$, etc.) will always indicate a codimension zero submanifold-with-boundary such that $B\cap\partial M=A$ and $\partial B\subseteq M$ is a compact properly embedded surface-with-boundary.
(Here and below we will, contrary to the usual meaning of $\partial$, use $\partial B$ to denote the boundary of $B$ as a subset of $M$ in the sense of point set topology.)

Given any $B\subseteq M$, we may perform any of the following operations:
\begin{itemize}
\item Removal of a neighborhood of a disk $(D^2,\partial D^2)\hookrightarrow(B,\partial B)$ with essential boundary.
\item Removal of a component of $(B,\partial B)$ which is diffeomorphic to $(B^3,S^2)$.
\item Removal of a component of $(B,\partial B)$ which is diffeomorphic to $(F\times I,F\times\partial I)$ or $(F\tildetimes I,F\tildetimes\partial I)$ for a closed surface $F$.
\end{itemize}
Such an operation, applied to either $B$ or $B^\complement$, will be called a compression, and $B$ is called incompressible if it admits no such operations.
If $B$ is incompressible, then it is $P^2$-irreducible (since $M$ is $P^2$-irreducible and $B^3$ contains no closed incompressible surfaces).

To continue exploring the properties of incompressible $B\subseteq M$, let us begin by recalling the following well known result.

\begin{lemma}\label{cylinderincompressible}
Every incompressible surface inside $F\times I$ with boundary $\partial F\times\{*\}$ is isotopic to $F\times\{*\}$.
The same holds for twisted $I$-bundles $F\tildetimes I$, though with the additional possibility $F\tildetimes\partial I$ when $F$ is closed.
\end{lemma}

\begin{proof}
Fix a triangulation of $F$, denoting its $1$-skeleton by $F^{(1)}$.
Put our unknown incompressible surface $G\subseteq F\times I$ with $\partial G=\partial F\times\{*\}$ into general position with respect to $F^{(1)}\times I$.
For any $1$-simplex $\sigma^1\subseteq F^{(1)}$, the curves comprising $G\cap(\sigma^1\times I)$ come in three types: arcs connecting points over the same endpoint of $\sigma^1$, arcs connecting points over different endpoints of $\sigma^1$, and circles.
Arcs of the first type which are innermost may be eliminated by isotoping $G$.
Since $G$ is incompressible, innermost circles may also be eliminated by isotoping $G$.
These simplification operations eventually terminate leaving only arcs of the second type.
This reduces us to the case $F=D^2$.

To treat the case $F=D^2$, put our unknown incompressible surface $G\subseteq F\times I$ with $\partial G=\partial F\times\{*\}$ in general position with respect to the family of planes $L_t\times I\subseteq D^2\times I$ where $L_t=\{t\}\times I\subseteq I\times I=D^2$.
If for any $t$ for which $L_t$ and $G$ are transverse, one of the circular components of $G\cap L_t$ is essential in $G$, we can produce a nontrivial compressing disk for $G$ by iteratively isotoping away innermost inessential interesctions using incompressibility of $G$.
It follows that such $t$ do not exist, and from this we may deduce that $G$ is a disk.
It then follows from Alexander's theorem \cite{alexander} that $G$ is isotopic to $F\times\{*\}$.
\end{proof}

\begin{corollary}
$B\subseteq M$ is incompressible iff $\partial B$ is incompressible, its components are pairwise non-isotopic rel boundary, and none of its components is the boundary of an embedded $F\tildetimes I\subseteq M$ for closed $F$.
\qed
\end{corollary}

We denote by $\sL(M;A)$ the collection of isotopy classes of incompressible $B\subseteq M$.
There is an obvious inclusion relation on $\sL(M;A)$, namely $\B\leq\B'$ iff there are representatives $B,B'\subseteq M$ of $\B$ and $\B'$ with $B\subseteq B'$.
The next result shows how to produce elements of $\sL(M;A)$ with given order properties.

\begin{lemma}\label{maximalcompression}
Starting with a given $B\subseteq M$, any sequence of compressions eventually terminates at an incompressible $\bar B\subseteq M$.
Furthermore, if $B\subseteq B'\subseteq M$ and $B'$ is incompressible, then any incompressible $\bar B$ obtained from $B$ by iterated compressions satisfies $[\bar B]\subseteq[B']$.
\end{lemma}

\begin{proof}
We just look at what the operations do to the compact properly embedded surface-with-boundary $\partial B$.
There are thus two types of operations: removing a component (or two) of $\partial B$ and performing a $2$-surgery along a simple closed curve inside $\partial B$.
Note that since $M$ is irreducible, a non-trivial compression disk has essential boundary, so the $2$-surgeries are all along essential curves.
It suffices to show that no compact surface-with-boundary admits an infinite sequence of such operations (component removals and $2$-surgeries).
In such a sequence of operations, if there are finitely many $2$-surgeries, there must also be finitely many component removals, since after all the $2$-surgeries are done, there are at most finitely many components as our surface always remains compact.
It thus suffices to show that there cannot be infinitely many $2$-surgeries.
This is clear, since each $2$-surgery increases the Euler characteristic by $2$, and non-trivial $2$-surgeries cannot create components of positive Euler characteristic, so the Euler characteristic cannot become arbitrarily large.
\end{proof}

When studying incompressible $B\subseteq M$, the fact that disjoint isotopic incompressible surfaces are parallel (see Waldhausen \cite[Corollary 5.5]{waldhausen} or Johannson \cite[Proposition 19.1]{johannson}) will be of essential use.
For example, this implies that:

\begin{proposition}\label{propcylinder}
Suppose $B,B'\subseteq M$ represent the same class in $\sL(M;A)$ and that $B\subseteq B'$, meeting only along $\partial A=\partial\partial B=\partial\partial B'$, transversely.
Then the region $B'\setminus B^\circ$ is a product $\partial B\times I$ pinched along $\partial\partial B\times I$ (which becomes $\partial A$).
\qed
\end{proposition}

\begin{corollary}
$\sL(M;A)$ is a poset.
\end{corollary}

\begin{proof}
If $\B\leq\B'\leq\B$, then we can find $B\subseteq B'\subseteq B''\subseteq M$ with $[B]=[B'']=\B$ and $[B']=\B'$.
Now apply Proposition \ref{propcylinder} and Lemma \ref{cylinderincompressible}.
\end{proof}

We now wish to show that $\sL(M;A)$ is a lattice.
The lattice property arises from the following fundamental result due to Freedman--Hass--Scott \cite[\S 7]{freedmanhassscott}.

\begin{theorem}\label{fhstheorem}
Let $M$ be a compact $P^2$-irreducible three-manifold-with-boundary, and let $S\subseteq\partial M$ be a multicurve all of whose components are essential.
There exist representatives $(F,\partial F)\hookrightarrow(M,S)$ of every isotopy class of properly embedded two-sided incompressible surface with boundary contained in $S$ which simultaneously realize all disjointness relations (with the caveat that a closed one-sided surface is allowed to be the representative of the isotopy class of the boundary of its tubular neighborhood).
\qed
\end{theorem}

The proof of Theorem \ref{fhstheorem} proceeds by choosing a Riemannian metric on $M$ which is `convex' near $\partial M$ in a suitable sense (the version of this assumption which is easiest to use from a technical standpoint is for the metric to be a product $\partial M\times[0,\varepsilon)$ near the boundary, however $\partial M$ being weakly mean convex would also be sufficient).
The methods of Douglas \cite{douglas}, Sacks--Uhlenbeck \cite{sacksuhlenbeck,sacksuhlenbeck2}, and Schoen--Yau \cite{schoenyau} show that area minimizing maps exist in $\pi_1$-injective homotopy classes, and the methods of Osserman \cite{osserman} and Gulliver \cite{gulliver} show that these maps are immersions.
Finally, the results of Freedman--Hass--Scott \cite{freedmanhassscott} show that these area minimizing immersions are in fact embeddings (or double covers of embedded one-sided surfaces), and Waldhausen \cite[Corollary 5.5]{waldhausen} or Johannson \cite[Proposition 19.1]{johannson} guarantee that homotopy classes and isotopy classes of two-sided incompressible surfaces in $P^2$-irreducible three-manifolds-with-boundary coincide; see also Hass--Scott \cite{hassscott}.
Analogous piecewise-linear methods are contained in Jaco--Rubinstein \cite{jacorubinstein}.

\begin{corollary}\label{minimalrealizeall}
Suppose $M=\overline M\setminus Y$ for a compact $P^2$-irreducible three-manifold-with-boundary $\overline M$ and a codimension zero submanifold-with-boundary $Y\subseteq\partial\overline M$.
There exist representatives $B\subseteq M$ of every element of $\sL(M;A)$ which simultaneously realize all order relations (i.e.\ $[B]\leq[B']$ iff $B\subseteq B'$).
\end{corollary}

\begin{proof}
Choose representatives $B\subseteq M$ whose boundary components are the representatives of Theorem \ref{fhstheorem} inside $\overline M$.
These realize all order relations by Lemma \ref{disjointrealizeorder}.
\end{proof}

\begin{lemma}\label{disjointrealizeorder}
Let $B,B'\subseteq M$ be incompressible, and suppose that for every pair of components $F\subseteq\partial B$ and $F'\subseteq\partial B'$, either $F$ and $F'$ are disjoint and not isotopic or $F=F'$.
If $[B]\leq[B']$ then $B\subseteq B'$.
\end{lemma}

\begin{proof}
Since $[B]\leq[B']$, there exists $B'_1$ isotopic to $B'$ with $B\subseteq B'_1$.
It suffices to show that we can isotope $B'_1$ to $B'$ while maintaining the property that $B\subseteq B'_1$.
To do this, we isotope the boundary components of $B'_1$ one by one onto the corresponding boundary components of $B'$.

Let us begin with the common components of $\partial B$ and $\partial B'$.
Let $\partial B\supseteq F=F'\subseteq\partial B'$ be such a component, and let $F'_1\subseteq\partial B'_1$ be the corresponding component.
Now $F'$ and $F'_1$ are parallel, and the region in between contains no other boundary components by Lemma \ref{cylinderincompressible}, so there is an evident isotopy of $B'_1$ moving $F'_1$ to $F'$, which preserves the containment $B\subseteq B'_1$.

Let us now consider corresponding components $F'\subseteq\partial B'$ and $F'_1\subseteq\partial B'_1$ which are not isotopic to a component of $\partial B$.
By assumption $F'_1\cap B=\varnothing$ and $F'\cap\partial B=\varnothing$.
There are now two possibilities: either $F'\subseteq B$ or $F'\cap B=\varnothing$.
The first possibility is in fact impossible: it would imply that $F'_1$ and $F'$ are disjoint, hence parallel, and thus isotopic to the component of $\partial B$ in between them by Lemma \ref{cylinderincompressible}.
We are thus in the second situation of $F'\cap B=\varnothing$.
Now $F'_1,F'\subseteq B^\complement$ are isotopic inside $M$, and we want to show that they are isotopic in $B^\complement$.
This holds because $\partial B^\complement$ is incompressible.
\end{proof}

\begin{proposition}
$\sL(M;A)$ is a lattice.
\end{proposition}

\begin{proof}
We first consider the case that $M$ is as in Corollary \ref{minimalrealizeall}.
Let $\B_i\in\sL(M;A)$ be a finite collection of elements, and choose their canonical representatives $B_i\subseteq M$ as in Corollary \ref{minimalrealizeall}.
Fix small inward perturbations $B_i^-\subseteq B_i$ so that their boundaries $\partial B_i^-$ are mutually transverse, and define $B:=\bigcup_iB_i^-\subseteq M$ (or rather as a smoothing of the boundary of this union).
Now suppose $\B'\in\sL(M;A)$ is larger than every $\B_i'$.
The canonical representative $B'\subseteq M$ therefore satisfies $B'\supseteq B_i$ for every $B_i$, hence $B'\supseteq B$.
Let $\B$ denote the class of any maximal compression of $B$ as in Lemma \ref{maximalcompression}.
We thus have $\B_i\leq\B$ and $\B\leq\B'$ for any upper bound $\B'$ of all $\B_i$.
In other words, $\B$ is a least upper bound for the $\B_i$.

We now reduce the case of general $M$ to that treated above, namely when $M$ is as in Corollary \ref{minimalrealizeall}.
Given any incompressible $B^-\subseteq B^+\subseteq M$ and small inward/outward pushoffs $B^-_\varepsilon\subseteq B^-$ and $B^+_\varepsilon\supseteq B^+$, there is an inclusion
\begin{equation}
\sL((B^+_\varepsilon)^\circ\setminus B^-_\varepsilon;A\cap\partial((B^+_\varepsilon)^\circ\setminus B^-_\varepsilon))\to\sL(M;A)
\end{equation}
given by ``union with $B^-_\varepsilon$'', which exhibits the former as the subset $\{\B:[B^-]\leq\B\leq[B^+]\}$ of the latter.
The former satisfies the hypothesis of Corollary \ref{minimalrealizeall}, and thus is a lattice.
Using the fact that any finite subset of $\sL(M;A)$ has upper and lower bounds (by Lemma \ref{maximalcompression}), the lattice property now follows in general.
\end{proof}

\section{Proof of the main result}

\subsection{Smith theory}\label{secsmith}

Smith theory relates the topology of a space $X$ with the topology of the fixed set $X^{\ZZ/p}$ of a $\ZZ/p$ action on $X$ (for $p$ a prime).
Smith theory was introduced by Smith \cite{smithI,smithII,smithIII}, and a detailed study was undertaken in Borel \cite{borel}.
We recall here the main results of Smith theory as formulated in Bredon \cite{bredonsheaftheory}.

\begin{definition}[{Homology manifolds \cite[p329 V.9.1]{bredonsheaftheory}}]
Let $L$ be a field.
An $L$-homology $n$-manifold is a locally compact Hausdorff space $X$ satisfying the following two properties:
\begin{itemize}
\item There exists $k<\infty$ such that $H^i_c(X,\mathcal F)=0$ for $i>k$ and any sheaf $\mathcal F$ of $L$ vector spaces on $X$ (compare \cite[II.16]{bredonsheaftheory}).
\item The sheafification of $U\mapsto\Hom(H^i_c(U,L),L)$ vanishes for $i\ne n$ and is locally constant with one-dimensional stalks for $i=n$ (compare \cite[V.3]{bredonsheaftheory}).
\end{itemize}
(Here $H_c^\ast$ denotes compactly supported sheaf cohomology.)
\end{definition}

\begin{theorem}[{\cite[p388 V.16.32]{bredonsheaftheory}}]\label{htopmanifold}
If $n\leq 2$, then any $L$-homology $n$-manifold is a topological $n$-manifold (not necessarily paracompact).
\qed
\end{theorem}

\begin{theorem}[{Local Smith Theory \cite[pp409--10 V.20.1, V.20.2]{bredonsheaftheory}}]\label{smithforhomologymflds}
For any action of $\ZZ/p$ on a $\ZZ/p$-homology $n$-manifold $X$, the fixed set $F=X^{\ZZ/p}$ is a disjoint union of open pieces $\{F_{(r)}\subseteq F\}_{0\leq r\leq n}$, where $F_{(r)}$ is a $\ZZ/p$-homology $r$-manifold.
Furthermore, if $F_{(r)}\ne\varnothing$ then $p(n-r)$ is even.
\qed
\end{theorem}

\begin{lemma}[{Alexander duality \cite[Lemma 3.3]{pardonhilbertsmith}}]\label{alexanderduality}
For any closed subset $X\subseteq\RR^n$, there is an isomorphism $\check H^\ast_c(X)=\tilde H_{n-1-\ast}(\RR^n\setminus X)$.
\qed
\end{lemma}

The following summarizes everything we will need from the results recalled above:

\begin{theorem}\label{propositionsmiththeory}
Let $\sigma:M\to M$ be a homeomorphism of a topological three-manifold $M$ of prime order.
The fixed set $M^\sigma$ is a disjoint union of open pieces $\{M^\sigma_{(r)}\}_{0\leq r\leq 3}$ where $M^\sigma_{(r)}$ is a topological $r$-manifold.
Moreover, $M^\sigma_{(2)}$ can be non-empty only when $p=2$, and in this case $\sigma$ reverses orientation near $M^\sigma_{(2)}$.
\end{theorem}

\begin{proof}
By Theorem \ref{smithforhomologymflds}, the fixed set $M^\sigma$ is a disjoint union of open pieces $M^\sigma_{(r)}$ each of which is a $\ZZ/p$-homology $r$-manifold for $r\leq 3$.
For $r\leq 2$, Theorem \ref{htopmanifold} implies that $M^\sigma_{(r)}$ is a topological $r$-manifold.
For $r=3$, we in fact have that $M^\sigma_{(3)}\subseteq M$ is an open subset (and thus \emph{a fortiori} a topological $3$-manifold); this follows from applying Lemma \ref{alexanderduality} to $M^\sigma_{(3)}\cap B\subseteq B$ for small open balls $\RR^3\cong B\subseteq M$ ($\tilde H_{-1}(B\setminus M^\sigma_{(3)})=\check H^3_c(B\cap M^\sigma_{(3)})$ is nonzero for small balls $B$ since $M^\sigma_{(3)}$ is a homology $3$-manifold, whence $B\subseteq M^\sigma_{(3)}$).
Theorem \ref{smithforhomologymflds} also ensures that $r=2$ can only happen when $p=2$.

It remains to show that the action reverses orientation near $M^\sigma_{(2)}$.
This is not stated explicitly in \cite{bredonsheaftheory}, so we show how to derive it.
The fundamental isomorphism underlying Smith theory is that for any $x\in F$, the restriction map
\begin{equation}
H^\ast_{\ZZ/p}(X,X\setminus x)\to H^\ast_{\ZZ/p}(F,F\setminus x)
\end{equation}
is an isomorphism in sufficiently large degrees (this ultimately follows from the fact that $\ZZ/p$ acts freely on the finite-dimensional space $X\setminus F$).
In our present situation, we have $p=2$ and $x\in F_{(r)}$, so $(X,X\setminus x)\simeq S^n$ and $(F,F\setminus x)\simeq S^r$.
Hence we have $H^\ast_{\ZZ/2}(F,F\setminus x)=H^{\ast-r}(\RR P^\infty)$ and $H^\ast_{\ZZ/2}(X,X\setminus x)=H^{\ast-n}(\RR P^\infty)$ or $H^{\ast-n}(\RR P^\infty,\nu)$ (where $\nu$ denotes the nontrivial local system on $\RR P^\infty$ with fiber $\ZZ$) according to the action of $\ZZ/2$ on orientations of $X$ at $x$ (these isomorphisms come from the spectral sequence $E_2^{p,q}=H^p(BG,H^q(Z))\Rightarrow H^{p+q}_G(Z)$, which in our cases $Z=(X,X\setminus x)$ or $(F,F\setminus x)$ has no further differentials since $H^\ast(Z)$ is concentrated in a single degree).
Now $H^\ast(\RR P^\infty)$ vanishes in (large) even degrees and $H^\ast(\RR P^\infty,\nu)$ vanishes in odd degrees, which in the present situation of $n-r$ odd implies that the action must reverse orientation at $x$.
\end{proof}

\subsection{Smoothing theory for three-manifolds}

The fundamental smoothing result for homeomorphisms of three-manifolds is the following:

\begin{theorem}\label{thmbinghomeoapproximation}
Every homeomorphism $\psi:M\to N$ between smooth three-manifolds is a uniform limit of diffeomorphisms $\tilde\psi:M\to N$.
If $\psi$ is a diffeomorphism (onto its image) over $\Nbd K$ for $K\subseteq M$ closed, then we may take $\tilde\psi=\psi$ over $\Nbd K$.
\qed
\end{theorem}

Theorem \ref{thmbinghomeoapproximation} is due to Moise \cite[Theorem 2]{moiseV} and Bing \cite[Theorem 8]{bingtriangulation}, both using bare-hands methods of point-set topology.
Alternative proofs can be found in Shalen \cite[Approximation Theorem]{shalen} (using methods of smooth three-manifold topology, such as the loop theorem of Papakyriakopoulos \cite{papakyriakopoulos}) and Hamilton \cite[Theorem 1]{hamilton} (using the torus trick of Kirby \cite{kirby,kirbysiebenmann}, also see Hatcher \cite{hatchertorus}).
An immediate corollary of Theorem \ref{thmbinghomeoapproximation} is:

\begin{corollary}\label{smoothstructureexists}
Every topological three-manifold-with-boundary $M$ has a smooth structure.
We may take this smooth structure to coincide with any given smooth structure over $\Nbd K$ for closed $K\subseteq M$.
\qed
\end{corollary}

Here is another corollary which we will need:

\begin{corollary}\label{latticehomeoact}
The lattice $\sL(M;A)$ from \S\ref{sectionlatticethreemanifold} carries a natural action of the group of homeomorphisms of $M$ which are diffeomorphisms over $\Nbd\partial M$ and preserve $A$.
\end{corollary}

\begin{proof}
To define the action of a given homeomorphism $\gamma:M\to M$ smooth near the boundary and preserving $A$, argue as follows.
Let $\B\in\sL(M;A)$, and fix two parallel representatives $B'\subseteq B\subseteq M$.
Now for any two smoothings $\tilde\gamma$ and $\tilde{\tilde\gamma}$ sufficiently close to $\gamma$, we have $\tilde\gamma(B')\subseteq\tilde{\tilde\gamma}(B)$ and $\tilde{\tilde\gamma}(B')\subseteq\tilde\gamma(B)$, which implies (using antisymmetry of $\leq$) that $\tilde\gamma(\B)=\tilde{\tilde\gamma}(\B)$.
Hence we may define $\gamma(\B):=\tilde\gamma(\B)$ for any diffeomorphism $\tilde\gamma$ sufficiently close to $\gamma$.
It is immediate from this definition that this is a group action preserving the partial order.
\end{proof}

We will also need the following taming result for embeddings of surfaces into three-manifolds:

\begin{theorem}\label{tamereembedding}
Every continuous proper embedding $\iota:F\hookrightarrow M$ of a surface into a three-manifold is a uniform limit of tame proper embeddings $\tilde\iota$.
If $\iota$ is tame over $\Nbd K$ for closed $K\subseteq F$, then we may take $\tilde\iota=\iota$ over $\Nbd K$.
\qed
\end{theorem}

Theorem \ref{tamereembedding} is due to Bing \cite[Theorems 7 and 8]{bingapproximation} (later generalized to arbitrary $2$-complexes in \cite[Theorem 5]{bingtriangulation}) on the way to the proof of Theorem \ref{thmbinghomeoapproximation}.
Building on this work, Craggs showed that the tame approximation $\tilde\iota$ is up to isotopy for $2$-complexes with no local cut points in \cite[Theorem 8.2]{craggsisotopy}.
We will only need this result for surfaces:

\begin{theorem}\label{tamereembeddingunique}
Fix a continuous proper embedding $\iota:F\hookrightarrow M$ of a surface into a three-manifold.
For every uniform neighborhood $U_\varepsilon$ of $\iota$, there exists a uniform neighborhood $U_\delta$ of $\iota$ such that for all pairs of tame proper embeddings $\iota_1,\iota_2:F\hookrightarrow M$ in $U_\delta$, there is an isotopy between $\iota_1$ and $\iota_2$ inside $U_\varepsilon$.
\qed
\end{theorem}

\subsection{Setting up the proof}

Given an action $G\acts M$, we consider the following open subsets of $M$:
\begin{itemize}
\item$M^\free\subseteq M$ denotes the set of points $x\in M$ with trivial stabilizer $G_x=1$.
\item$M^\refl\subseteq M$ denotes the set of points $x\in M$ for which either $G_x=1$ or $G_x=\ZZ/2$ and $x\in M^{G_x}_{(2)}$ (i.e.\ $M^{G_x}$ is locally a surface near $x$).
The closed locus $F^\refl\subseteq M^\refl$ with isotropy group $\ZZ/2$ is a topological surface, possibly wildly embedded.
\end{itemize}
We have obvious inclusions $M^\free\subseteq M^\refl\subseteq M$.
An action $G\acts M$ is called \emph{generically free} iff $M^\free\subseteq M$ is dense.
By Theorem \ref{propositionsmiththeory}, an action is generically free as long as no nontrivial element of $G$ acts trivially on an entire connected component of $M$.

\begin{lemma}\label{connectedeffective}
The general case of Theorem \ref{main} follows from the special case of $M$ connected and $\varphi:G\acts M$ generically free.
\end{lemma}

\begin{proof}
We reduce to the case of $M$ connected as follows.
First, by decomposing $\pi_0(M)$ into $\varphi(G)$-orbits, we reduce to the case that $G$ acts transitively on $\pi_0(M)$.
Next, fix a connected component $M_0\subseteq M$, so we have $M=\bigsqcup_{g\in G/\Stab([M_0])}gM_0$.
Now the action of $G$ on $M$ is determined uniquely by the data of (1) the action of $\Stab([M_0])$ on $M_0$ and (2) the homeomorphisms $M_0\to gM_0$ provided by any fixed choice of representatives in $G$ of the nontrivial elements of $G/\Stab([M_0])$.
The homeomorphisms (2) can be approximated uniformly by diffeomorphisms by Theorem \ref{thmbinghomeoapproximation}, and smoothing the action (1) of $\Stab([M_0])$ on $M_0$ requires precisely the connected case of Theorem \ref{main}.

Now consider $\varphi:G\acts M$ with $M$ connected.
By invariance of domain, if $M^{\varphi(g)}_{(3)}\ne\varnothing$ then $M^{\varphi(g)}_{(3)}=M$, that is $g$ acts trivially on $M$.
Thus the action of $G/\ker\varphi$ on $M$ is generically free, and it suffices to smooth this action.
\end{proof}

The complement of $M^\refl$ is essentially one-dimensional:

\begin{lemma}\label{reflcomplement}
If $G\acts M$ is generically free, then
\begin{equation}
M\setminus M^\refl=\bigcup_{\begin{smallmatrix}g\in G\setminus\{1\}\\g^p=1\end{smallmatrix}}M^g_{(0)}\cup M^g_{(1)}.
\end{equation}
\end{lemma}

\begin{proof}
The non-trivial direction is to show that if $x\in M\setminus M^\refl$ then it is in the right hand side above.
If $x\notin M^\refl$, then either $G_x=\ZZ/2$ and $x\in M^{G_x}_{(r)}$ for $r\leq 1$ (in which case $x$ is by definition contained in the right hand side above), or $\left|G_x\right|>2$.
In the latter case $\left|G_x\right|>2$, the subgroup of $G_x$ which preserves orientation at $x$ (which has index at most $2$) is non-trivial and hence contains some element of prime order, so $x$ is in the right hand side by Theorem \ref{propositionsmiththeory}.
\end{proof}

\subsection{Smoothing over the free locus}

\begin{proposition}\label{smoothoverfree}
Every continuous action $\varphi:G\acts M$ of a finite group on a smooth three-manifold is a uniform limit of actions $\tilde\varphi:G\acts M$ which are smooth over $M^{\tilde\varphi\mhyphen\free}=M^{\varphi\mhyphen\free}$ and coincide with $\varphi$ over the complement.
If $\varphi$ is smooth over $\Nbd K$ for $K\subseteq M$ closed and $\varphi(G)$-invariant, then we may take $\tilde\varphi=\varphi$ over $\Nbd K$.
\end{proposition}

\begin{proof}
The quotient $M^{\varphi\mhyphen\free}/\varphi(G)$ is a topological three-manifold, which by Corollary \ref{smoothstructureexists} has a smooth structure.
Denote by $(M^{\varphi\mhyphen\free})^s$ the pullback smooth structure on $M^{\varphi\mhyphen\free}$, so now $\varphi:G\acts(M^{\varphi\mhyphen\free})^s$ is smooth.
Now the identity map $\id:M^{\varphi\mhyphen\free}\to(M^{\varphi\mhyphen\free})^s$ is a homeomorphism, which by Theorem \ref{thmbinghomeoapproximation} can be approximated by a diffeomorphism $\alpha:M^{\varphi\mhyphen\free}\to(M^{\varphi\mhyphen\free})^s$.
Thus the action $\alpha^{-1}\varphi\alpha:G\acts M^{\varphi\mhyphen\free}$ is smooth.
Theorem \ref{thmbinghomeoapproximation} allows us to take $\alpha$ to extend continuously to a homeomorphism $\bar\alpha:M\to M$ acting as the identity on the complement of $M^{\varphi\mhyphen\free}$.
Hence $\tilde\varphi:=\bar\alpha^{-1}\varphi\bar\alpha:G\acts M$ is the desired approximation of $\varphi$.
\end{proof}

\subsection{Smoothing over the tame reflection locus}

Let $F^\refl_\tame\subseteq F^\refl$ denote the open subset where $F^\refl$ is tamely embedded inside $M^\refl$, and let $M^\trefl:=M^\free\cup F^\refl_\tame$.

\begin{proposition}\label{smoothovertrefl}
Every continuous action $\varphi:G\acts M$ of a finite group on a smooth three-manifold is a uniform limit of actions $\tilde\varphi:G\acts M$ which are smooth over $M^{\varphi\mhyphen\trefl}\subseteq M^{\tilde\varphi\mhyphen\trefl}$ and coincide with $\varphi$ over the complement of $M^{\varphi\mhyphen\trefl}$.
If $\varphi$ is smooth over $\Nbd K$ for $K\subseteq M$ closed and $\varphi(G)$-invariant, then we may take $\tilde\varphi=\varphi$ over $\Nbd K$.
\end{proposition}

\begin{proof}
The proof of Proposition \ref{smoothoverfree} applies without significant change.
The quotient $M^{\varphi\mhyphen\trefl}/\varphi(G)$ is now a topological three-manifold-with-boundary, again smoothable by Corollary \ref{smoothstructureexists}.
Choosing arbitrarily a (germ of) smooth boundary collar for $M^{\varphi\mhyphen\trefl}/\varphi(G)$ provides a lift of this smooth structure to $M^{\varphi\mhyphen\trefl}$, and the rest of the proof is the same.
\end{proof}

\subsection{Taming the reflection locus}

Let $F^\refl_\wild:=F^\refl\setminus F^\refl_\tame$ denote the closed subset where $F^\refl$ is wildly embedded inside $M^\refl$.

\begin{proposition}\label{tamerefl}
Every continuous action $\varphi:G\acts M$ of a finite group on a smooth three-manifold is a uniform limit of actions $\tilde\varphi:G\acts M$ for which $F^{\tilde\varphi\mhyphen\refl}_\wild$ is contained in the $1$-skeleton of a $G$-invariant triangulation of $F^{\tilde\varphi\mhyphen\refl}$ and $M^{\varphi\mhyphen\refl}=M^{\tilde\varphi\mhyphen\refl}$.
Moreover, we may take $\tilde\varphi=\varphi$ except over a neighborhood of $F^{\varphi\mhyphen\refl}_\wild$ inside $M^{\varphi\mhyphen\refl}$.
\end{proposition}

\begin{proof}
Fix a very fine $G$-equivariant triangulation of $F^{\varphi\mhyphen\refl}$ (note that $G$ acts with constant stabilizer on $F^{\varphi\mhyphen\refl}$ and that $F^\refl_\wild$ is $G$-invariant).
It suffices to describe how to modify $\varphi$ in an arbitrarily small neighborhood of each $G$-orbit of open $2$-simplices intersecting the wild locus (then just do all of these modifications simultaneously).
Note that $F^{\varphi\mhyphen\refl}\subseteq M$ is not generally a closed subset, and hence there may be infinitely many such $2$-simplices in any neighborhood of some points of $M$, but that this is not an issue.

Let a $G$-orbit of $2$-simplices inside $F^{\varphi\mhyphen\refl}$ be given.
Fix an open $2$-simplex $U\subseteq F^{\varphi\mhyphen\refl}$ in this orbit, with stabilizer an involution $\sigma\in G$.
It suffices to modify the action of $\sigma$ in a neighborhood of $U$ (choosing coset representatives for $G/\sigma$ as in the proof of Lemma \ref{connectedeffective} extends this to a modification of the action of $G$ near the union of translates of $U$).

To find the desired new action of $\sigma$ locally near $U$, we follow the argument of Craggs \cite[Theorem 3.1]{craggs}.
Note that, as a consequence of Alexander duality (see Lemma \ref{alexanderduality}), $U$ divides $M$ locally into two `sides' and $\sigma$ exchanges these two sides since it reverses orientation near $U$.
Let $\iota:U\hookrightarrow M$ be the identity map embedding, and let $\tilde\iota:U\hookrightarrow M$ be a tame reembedding as produced by Bing's Theorem \ref{tamereembedding}.
Now $\sigma\circ\tilde\iota:U\to M$ is another tame reembedding, so by Craggs' Theorem \ref{tamereembeddingunique}, there is a uniformly small ambient isotopy $\{h_t:M\to M\}_{t\in[0,1]}$ supported near $U$ from $h_0=\id$ to a homeomorphism $h_1$ such that $h_1\circ\tilde\iota=\sigma\circ\tilde\iota$.
Now on one side of $\tilde\iota(U)$, we define $\tilde\sigma:=h_1^{-1}\circ\sigma$, and on the other side we take its inverse $\sigma\circ h_1$.
This is a new involution $\tilde\sigma$, coinciding with $\sigma$ outside a neighborhood of $U$, and with fixed set $\tilde\iota(U)$ which is by definition tame.
\end{proof}

\begin{remark}
Given a relative version of Craggs' Theorem \ref{tamereembeddingunique}, in the sense that the isotopies could be made to be constant over a locus where $\iota=\iota_1=\iota_2$ (this may even be proved in \cite{craggsisotopy}), we could iterate the process in the above proof over a neighborhood of the $1$-simplices and then the $0$-simplices, thus taming the entire $F^\refl$.
This is a moot point, however, since the weaker statement of Proposition \ref{tamerefl} is all that is needed to prove Theorem \ref{main}.
\end{remark}

\subsection{Smoothing over the remainder}

It is here that the results of \S\ref{secnielsen} and \S\ref{sectionlattice} are put to use.

\begin{theorem}\label{smoothoverremainder}
Every generically free continuous action $\varphi:G\acts M$ of a finite group on a smooth three-manifold which is smooth over $M^{\varphi\mhyphen\refl}$ minus the $1$-skeleton of a $G$-invariant triangulation of $F^{\varphi\mhyphen\refl}$ is a uniform limit of smooth actions $\tilde\varphi:G\acts M$.
If $\varphi$ is smooth over $\Nbd K$ for $K\subseteq M$ closed and $\varphi(G)$-invariant, then we may take $\tilde\varphi=\varphi$ over $\Nbd K$.
\end{theorem}

\begin{proof}
Let $X\subseteq M$ be the (necessarily closed and $G$-invariant) locus where $\varphi$ fails to be smooth.
By hypothesis, $X$ is contained within $M\setminus M^{\varphi\mhyphen\refl}$ union the $1$-skeleton of a $G$-invariant triangulation of $F^{\varphi\mhyphen\refl}$.
Appealing to Lemma \ref{reflcomplement} on the structure of $M\setminus M^\refl$, we conclude that $X/G$ has Lebesgue covering dimension at most $1$.

Let $M_0\subseteq M$ be a small $G$-invariant closed neighborhood of $X$ with smooth boundary, and let $M_0=\bigcup_iU_i$ be a locally finite $G$-equivariant (i.e.\ the action of $G$ permutes the $U_i$) open cover by small open subsets of $M_0$, whose nerve has dimension at most $1$ (i.e.\ all triple intersections $U_i\cap U_j\cap U_k$ for distinct $i,j,k$ are empty) and such that $G$ does not exchange any pair $(i,j)$ with $U_i\cap U_j\ne\varnothing$ (i.e.\ the action of $G$ on the nerve of the cover does not invert any edge).

\newcommand\Uz{(0,0)circle(2)}
\newcommand\Ua{(6,0)circle(1.7)}
\newcommand\Ub{(4,0)ellipse(1.8 and 1.4)}
\newcommand\Uc{(2,0)ellipse(1 and 1.5)}
\newcommand\Ud{(-2,-2)circle(1.4)}
\newcommand\Ue{(-3.3,-3.7)circle(1.3)}
\newcommand\Uf{(0,2.4)ellipse(1.8 and 1)}
\newcommand\Ug{(0,3.5)ellipse(1.9 and 0.8)}
\newcommand\Uh{(0,4.8)ellipse(1.8 and 0.7)}
\newcommand\Udasheddotted{
\draw[thick,dotted]\Uz;
\draw[thick,dashed]\Uc;
\draw[thick,dotted]\Ub;
\draw[thick,dotted]\Ua;
\draw[thick,dashed]\Uf;
\draw[thick,dotted]\Ug;
\draw[thick,dashed]\Uh;
\draw[thick,dashed]\Ud;
\draw[thick,dotted]\Ue;}
\newcommand\mboundone{(1,5)to[out=-90,in=90](0.6,4)to[out=-90,in=90](1.5,3.7)to[out=-90,in=90](1.4,3)to[out=-90,in=90](1.5,2.5)to[out=-90,in=90](1,1.3)to[out=-90,in=180](1.3,1)to[out=0,in=180](2,0.9)to[out=0,in=180](4,0.8)to[out=0,in=180](6,1)}
\newcommand\mboundtwo{(6,-1)to[out=-180,in=0](4,-0.8)to[out=-180,in=0](2,-0.9)to[out=-180,in=20](1.3,-1)to[out=200,in=20](0.5,-1.6)to[out=200,in=45](-0.7,-1.7)to[out=225,in=45](-1.3,-2.7)to[out=225,in=45](-2.2,-3.3)to[out=225,in=45](-2.5,-3.8)}
\newcommand\mboundthree{(-3.5,-2.8)to[out=45,in=225](-3.1,-2.4)to[out=45,in=225](-2.7,-1.3)to[out=45,in=240](-1.7,-0.7)to[out=60,in=240](-1.6,0.7)to[out=60,in=-90](-0.9,1.7)to[out=90,in=-90](-1.3,3.5)to[out=90,in=-90](-0.6,4.3)to[out=90,in=-90](-1.2,5)}


\begin{figure}[htbp]
\centering
\begin{tikzpicture}[scale=\textwidth/22cm]
\draw(0,0)--(0,1)--(0.2,1.2)--(0,2)--(-0.2,1.8)--(-0.1,2.3)--(0.5,1.5)--(0.8,1.8)--(0.5,1.8)--(0.8,2.1)--(0.5,2.1)--(0.8,2.4)--(0.5,2.4)--(1.1,3)--(0.5,2.7)--(0,3)--(0.7,3.2)--(0,3.4)--(0,4.5)--(0,5.1);
\draw(0,0)--(1,0.3)--(0.7,0)--(1.2,-0.1)--(1.1,0)--(1.4,0.3)--(1.7,0)--(1.5,-0.2)--(1.5,-0.4)--(1.7,-0.6)--(2.3,0)--(2.7,0)--(2.5,-0.2)--(2.3,-0.7)--(5,0)--(5.5,0)--(6.1,0);
\draw(0,0)--(-1,-0.4)--(-0.6,-0.7)--(-0.6,-0.4)--(-0.2,-0.7)--(-0.9,-1.6)--(-1.2,-1.4)--(-2,-2)--(-3,-3.4)--(-3.3,-3.7);
\node at(6.6,0){$X$};
\begin{scope}
\clip(6,5)--(-3.7,5)--(-3.7,-4.1)--(6,-4.1)--cycle;
\Udasheddotted
\node at(2,-2){$U_i'$};
\end{scope}
\end{tikzpicture}
\quad
\begin{tikzpicture}[scale=\textwidth/22cm]
\begin{scope}
\clip\mboundone--\mboundtwo--\mboundthree--cycle;
\Udasheddotted
\end{scope}
\draw\mboundone;
\draw\mboundtwo;
\draw\mboundthree;
\node at(3,3){$M_0$};
\end{tikzpicture}
\caption{Left: The locus $X$ and the open covering $X\subseteq\bigcup_iU_i'$.  Right: the neighborhood $X\subseteq M_0\subseteq\bigcup_iU_i'$ and the covering $M_0=\bigcup_iU_i$ defined by $U_i:=U_i'\cap M_0$.}\label{figurecoverconstruction}
\end{figure}
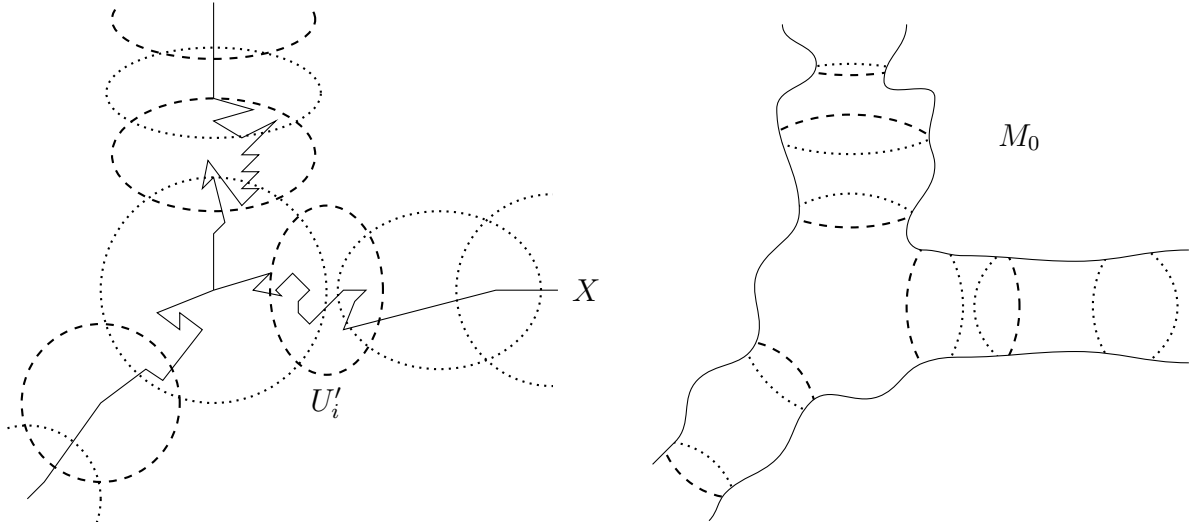

To construct $M_0$ and this open cover, argue as follows (see Figure \ref{figurecoverconstruction}).
Choose a very fine locally finite closed cover of $X/G$ whose nerve has dimension at most $1$, and pull it back to $X$.
This produces a $G$-invariant cover of $X$ (i.e.\ each set in the cover is stabilized by $G$).
By further breaking up each of these inverse images into finitely many disjoint pieces permuted by $G$, we obtain an arbitrarily fine locally finite $G$-equivariant closed cover $X=\bigcup_iV_i$ with nerve of dimension at most $1$.
There may be some bad pairs $(i,j)$ with $V_i\cap V_j\ne\varnothing$ and $G$ exchanging $i$ and $j$.
In this case, we may add a small neighborhood of $V_i\cap V_j$ to the cover and shrink $V_i$ and $V_j$ accordingly (the effect of this operation on the nerve of the covering is to add a vertex at the middle of the flipped edge $(i,j)$).
This operation takes place in a small neighborhood of $V_i\cap V_j$, so we may simply do it to all bad pairs simultaneously.
Now choose open neighborhoods $U_i'$ of $V_i$ inside $M$ such that $U_i'\cap U_j'\ne\varnothing$ only if $V_i\cap V_j\ne\varnothing$.
Finally, pick any $G$-invariant $M_0\subseteq M$, closed with smooth boundary, containing $X$, and contained inside $\bigcup_iU_i'$, and set $U_i:=M_0\cap U_i'$.

So that we may apply the results of \S\ref{sectionlattice}, we further specify the construction of $M_0$ and the open cover $M_0=\bigcup_iU_i$ as follows.
Let us call a bounded open subset $U\subseteq\RR^3$ \emph{saturated} iff its complement is connected (equivalently $H_2(U)=0$).
Every bounded open set $U\subseteq\RR^3$ is contained in a unique minimal bounded saturated $U^+\subseteq\RR^3$, obtained by adding to $U$ the bounded component of $\RR^3\setminus F$ for every embedded surface $F\subseteq U$.
If $U$ is saturated, then it is irreducible by Alexander's theorem \cite{alexander}.
Note that if $U$ and $V$ are both saturated, then so is their intersection $U\cap V$.
The notion of being saturated also makes sense (and the above discussion continues to apply) for open subsets of $M$ of small diameter.

Let us now argue that we can choose the $U_i$ (and hence also their pairwise intersections $U_i\cap U_j$) to be saturated (hence, in particular, irreducible).
First, fix an open covering of $X$ by small open balls $B_\alpha\subseteq M$.
Let us require that the covering $V_i$ of $X$ is chosen such that each $V_i$ is contained within some $B_\alpha$.
Since $X$ has covering dimension at most $1$, so does each $V_i$, implying that $\check H^2(V_i)=0$, and hence removing $V_i$ does not disconnect $B_\alpha$.
It follows that $V_i$ has arbitrarily small neighborhoods $U_i'$ (contained in a compact subset of the same $B_\alpha$) whose removal does not disconnect $B_\alpha$, i.e.\ $U_i'$ is saturated.
We now turn to the choice of $M_0$.
Beginning with an arbitrary choice of $M_0$ as above (i.e.\ $G$-invariant, closed with smooth boundary, containing $X$, and contained inside $\bigcup_iU_i'$), consider those boundary components $F$ of $M_0$ which lie entirely inside a given $U_i'$.
Since every $U_i'$ is saturated, we may add to $M_0$ the compact region inside $U_i'$ bounded by each such $F$, and the condition that $M_0\subseteq\bigcup_iU_i'$ is preserved (as are the other conditions).
It now follows that $U_i=U_i'\cap M_0$ is also saturated: the contrary would be the existence of a closed surface $F\subseteq U_i'\cap M_0$ bounding a compact region in $B_\alpha$ not both entirely inside $U_i'$ and inside $M_0$, however by construction these do not exist.

We now consider the lattices $\sL(\partial(U_i\cap U_j))$ (recall \S\ref{sectionlatticesurface}).
We claim that there exists a $G$-invariant collection of elements $\A_{ij}\in\sL(\partial(U_i\cap U_j))$ with $\A_{ij}=\A_{ji}^\complement$ such that $\A_{ij}$ contains the $U_i$ end of $\partial(U_i\cap U_j)$ (note that the boundary $(\overline{U_i\cap U_j})\setminus(U_i\cap U_j)$ is the open disjoint union of its intersection with $U_i$ and its intersection with $U_j$).
To see this, start with any not necessarily $G$-invariant collection of $\A_{ij}=\A_{ji}^\complement$.
Considering all $G$-translates, we get a collection of finite multisets $S_{ij}\subseteq\sL(\partial(U_i\cap U_j))$.
Now choose a $G$-invariant preferred order $(i,j)$ for every unordered pair of indices with $U_i\cap U_j$ non-empty (this is possible since $G$ doesn't swap any such pair of indices).
For $(i,j)$ in this preferred order, define $\A_{ij}$ to be the least upper bound of $S_{ij}$ (and $\A_{ji}$ to be its complement, i.e.\ the greatest lower bound of $S_{ji}$).
This is the desired collection $\A_{ij}$.
We now fix representatives $A_{ij}\subseteq\partial(U_i\cap U_j)$ of $\A_{ij}$ which satisfy $A_{ij}=A_{ji}^\complement$ and $G$-invariance on the nose rather than only up to isotopy (for instance, we could choose $\partial A_{ij}$ to be geodesics in a $G$-invariant hyperbolic metric on $\partial(U_i\cap U_j)$).

We now consider the lattices $\sL(U_i\cap U_j;A_{ij})$ (recall \S\ref{sectionlatticethreemanifold}), and we claim that there exists a $G$-invariant collection of elements $\B_{ij}\in\sL(U_i\cap U_j;A_{ij})$ with $\B_{ij}=\B_{ji}^\complement$ such that $\B_{ij}$ contains the $U_i$ end of $U_i\cap U_j$ (recall from Corollary \ref{latticehomeoact} that $G$ does indeed act on these lattices).
Indeed, such $\B_{ij}$ may be obtained using the same procedure used above to construct the $\A_{ij}$.
Fix representatives $B_{ij}\subseteq U_i\cap U_j$ with $B_{ij}=B_{ji}^\complement$ which are $G$-invariant in a neighborhood of $\partial(U_i\cap U_j)$ (but \emph{not} necessarily globally $G$-invariant).

\newcommand\patha{(5,-1.5)to[out=90,in=-90](4.7,-0.4)to[out=90,in=-90](5.1,0.2)to[out=90,in=-90](5,1.5)}
\newcommand\pathb{(2.7,-1)to[out=90,in=-90](2.5,0.2)to[out=90,in=-90](2.8,1.2)}
\newcommand\pathc{(1.5,-1.5)to[out=90,in=-90](1.7,0)to[out=90,in=-90](1.4,1.5)}
\newcommand\pathd{(-0.7,-1.9)to[out=90,in=0](-1.2,-1.2)to[out=180,in=-45](-2,-0.5)}
\newcommand\pathe{(-2.1,-3.4)to[out=135,in=-45](-2.5,-3)to[out=135,in=-45](-3.5,-2.3)}
\newcommand\pathf{(1.5,1.6)to[out=180,in=0](0,1.7)to[out=180,in=0](-1.5,1.6)}
\newcommand\pathg{(2,2.9)to[out=180,in=0](0.3,3.2)to[out=180,in=0](-0.2,2.9)to[out=180,in=0](-2,3)}
\newcommand\pathh{(1,4.2)to[out=180,in=0](-1,4.2)}

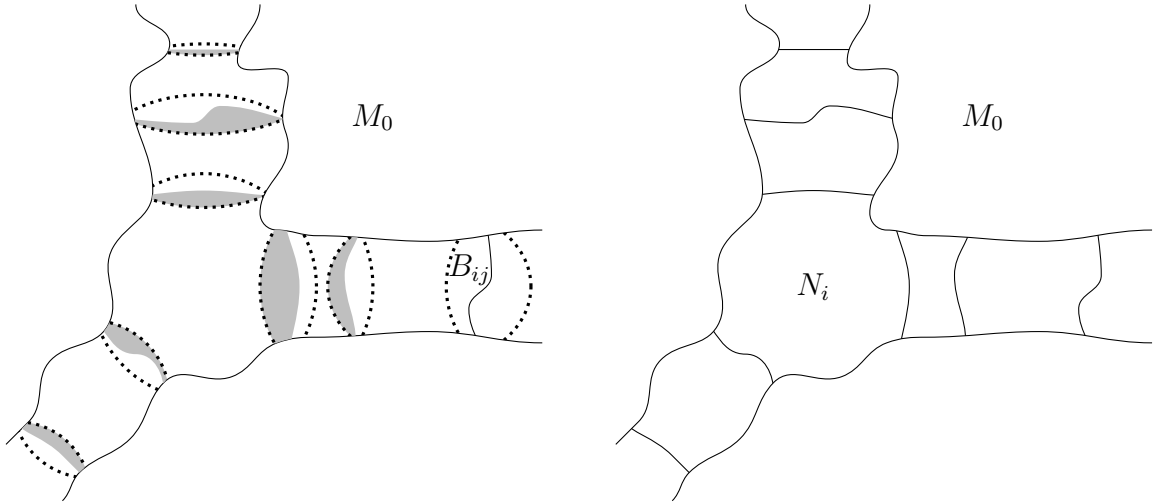
\begin{figure}[htbp]
\centering
\begin{tikzpicture}[scale=\textwidth/22cm]
\draw\mboundone;
\draw\mboundtwo;
\draw\mboundthree;
\node at(3,3){$M_0$};
\begin{scope}
\clip\mboundone--\mboundtwo--\mboundthree--cycle;
\begin{scope}\clip\Ua;\draw\patha;\end{scope}\node at(4.71,0.3){$B_{ij}$};
\begin{scope}\clip\Ub;\path[fill=lightgray]\pathb--(0,0)--cycle;\end{scope}
\begin{scope}\clip\Uc;\path[fill=lightgray]\pathc--(0,0)--cycle;\end{scope}
\begin{scope}\clip\Ud;\path[fill=lightgray]\pathd--(0,0)--cycle;\end{scope}
\begin{scope}\clip\Ue;\path[fill=lightgray]\pathe--(0,0)--cycle;\end{scope}
\begin{scope}\clip\Uf;\path[fill=lightgray]\pathf--(0,0)--cycle;\end{scope}
\begin{scope}\clip\Ug;\path[fill=lightgray]\pathg--(0,0)--cycle;\end{scope}
\begin{scope}\clip\Uh;\path[fill=lightgray]\pathh--(0,0)--cycle;\end{scope}
\draw[very thick, dotted]\Uz;
\draw[very thick, dotted]\Ua;
\draw[very thick, dotted]\Ub;
\draw[very thick, dotted]\Uc;
\draw[very thick, dotted]\Ud;
\draw[very thick, dotted]\Ue;
\draw[very thick, dotted]\Uf;
\draw[very thick, dotted]\Ug;
\draw[very thick, dotted]\Uh;
\end{scope}
\end{tikzpicture}
\qquad
\begin{tikzpicture}[scale=\textwidth/22cm]
\begin{scope}
\clip\mboundone--\mboundtwo--\mboundthree--cycle;
\draw\patha;
\draw\pathb;
\draw\pathc;
\draw\pathd;
\draw\pathe;
\draw\pathf;
\draw\pathg;
\draw\pathh;
\end{scope}
\node at(0,0){$N_i$};
\draw\mboundone;
\draw\mboundtwo;
\draw\mboundthree;
\node at(3,3){$M_0$};
\end{tikzpicture}
\caption{Left: the submanifolds $B_{ij}\subseteq U_i\cap U_j$.  Right: the resulting partition of $M_0$ into submanifolds $N_i$.}\label{figurepartition}
\end{figure}

We now consider the partition $M_0=\bigcup_iN_i$ where $N_i:=U_i\setminus\bigcup_jB_{ji}^\circ$ (informally, we cut $M_0$ along $\partial B_{ij}$; see Figure \ref{figurepartition}), and we argue that the given strict action $\varphi:G\acts M_0$ can be deformed (over compact subsets of $\bigsqcup_{i,j}U_i\cap U_j$ and away from $\partial M_0$) to a homotopy action by homotopy equivalences $\bar\varphi:G\hacts M_0$ which preserves the partition $M_0=\bigcup_iN_i$ and agrees with $\varphi$ near $\partial M_0$.
In the proof of Theorem \ref{nielsenall}, we cut a homotopy action along the tori of the JSJ decomposition, and we will use a similar strategy here.
First, use Bing--Moise to approximate $\varphi^0$ by diffeomorphisms on the pairwise intersections $U_i\cap U_j$.
These approximating diffeomorphisms are homotopic (via a small homotopy) to the original $\varphi^0$, so we may deform $\varphi$ to obtain $\bar\varphi$ (coinciding with $\varphi$ away from the $U_i\cap U_j$) for which $\bar\varphi^0$ are diffeomorphisms on $U_i\cap U_j$ (note, however, that this comes at the cost that the higher components of $\bar\varphi$ now may only be homotopy equivalences on $U_i\cap U_j$ rather than homeomorphisms).
Now we may further deform $\bar\varphi$ inside $U_i\cap U_j$ first so that it stabilizes $\partial N_{ij}$ (using incompressibility of $\partial N_{ij}$) and then so that it restricts to a strict action on $\partial N_{ij}$ using Corollary \ref{surfacerelboundary}.
Finally, using the same argument from the proof of Theorem \ref{nielsenall}, we deform $\bar\varphi$ (relative $\partial N_{ij}$) by induction on $k\geq 1$ so that it preserves the $N_{ij}$ as well.
We thus have a homotopy action $\bar\varphi$ which coincides with $\varphi$ away from $U_i\cap U_j$ and which preserves $N_{ij}$ and acts strictly on $\partial N_{ij}$.
Hence $\bar\varphi$ determines via cutting a homotopy action by homotopy equivalences on $\bigsqcup_iN_i$ strict near the boundary.
On those components of $\bigsqcup_iN_i$ which are not $B^3$, we may use Theorem \ref{nielsenall} and Corollary \ref{relsurfaceboundary} to homotope $\bar\varphi$ rel boundary to be strict.
On those components of $\bigsqcup_iN_i$ which are $B^3$, we recall that every strict action of a finite group on $\bigsqcup\partial B^3=\bigsqcup S^2$ preserves some spherical metric \cite[Theorem 2.4]{scottgeometries}, and thus extends to a strict action on $\bigsqcup B^3$.
\end{proof}

\begin{proof}[Proof of Theorem \ref{main}]
By Lemma \ref{connectedeffective} it is enough to treat the generically free case.
Using Proposition \ref{tamerefl} we tame $F^\refl$ away from a $1$-skeleton.
Then using Proposition \ref{smoothovertrefl}, we smooth the action over $M^\trefl$.
Finally, Theorem \ref{smoothoverremainder} smooths the rest.
\end{proof}

\bibliographystyle{amsalpha}
\bibliography{tamefinite}
\addcontentsline{toc}{section}{References}

\end{document}